\documentclass[11pt]{article}

\usepackage{paralist, amsmath, amsthm, amsfonts, amssymb, xy, color,
  mathrsfs}
\usepackage[margin=1in]{geometry} 
\usepackage{hyperref}

\input xy
\xyoption{all}

\numberwithin{equation}{section}
\newtheorem{theorem}[equation]{Theorem}
\newtheorem*{thm}{Theorem}
\newtheorem{proposition}[equation]{Proposition}
\newtheorem{lemma}[equation]{Lemma}
\newtheorem{corollary}[equation]{Corollary}

\theoremstyle{definition}
\newtheorem{defn}[equation]{Definition}
\newtheorem{eg}[equation]{Example}
\newtheorem{rmk}[equation]{Remark}
\newenvironment{definition}{\begin{defn}}{\hfill $\blacksquare$ \end{defn}}
\newenvironment{example}{\begin{eg}}{\hfill $\blacksquare$ \end{eg}}
\newenvironment{remark}{\begin{rmk}}{\hfill $\blacksquare$ \end{rmk}}

\newcommand{\N}{\mathbf{N}}

\newcommand{\Z}{\mathbf{Z}}

\renewcommand{\phi}{\varphi}

\newcommand{\eps}{\varepsilon}

\renewcommand{\tilde}[1]{\widetilde{#1}}
\newcommand{\ol}[1]{\overline{#1}}
\newcommand{\ul}[1]{\underline{#1}}
\newcommand{\DS}{\displaystyle}

\newcommand{\arxiv}[1]{\href{http://arxiv.org/abs/#1}{{\tt arXiv:#1}}}

\makeatletter
\def\Ddots{\mathinner{\mkern1mu\raise\p@
\vbox{\kern7\p@\hbox{.}}\mkern2mu
\raise4\p@\hbox{.}\mkern2mu\raise7\p@\hbox{.}\mkern1mu}}
\makeatother

\DeclareMathOperator{\im}{image} 
\renewcommand{\hom}{\operatorname{Hom}} 
\DeclareMathOperator{\ext}{Ext} 
\newcommand{\GL}{\mathbf{GL}}
\newcommand{\SL}{\mathbf{SL}}
\newcommand{\Sp}{\mathbf{Sp}}

\newcommand{\SO}{\mathbf{SO}}

\newcommand{\Sym}{\mathrm{Sym}}

\DeclareMathOperator{\pdim}{pdim}
\DeclareMathOperator{\idim}{idim}
\DeclareMathOperator{\gldim}{gldim}

\newcommand{\bL}{\mathbf{L}}

\newcommand{\QQ}{\mathscr{Q}}
\newcommand{\RR}{\mathcal{R}}
\newcommand{\bS}{\mathbf{S}}

\DeclareMathOperator{\Rep}{Rep}

\DeclareMathOperator{\SRep}{SRep}

\newcommand{\SG}{\mathbf{SG}}
\DeclareMathOperator{\SI}{SI}
\DeclareMathOperator{\SSI}{SSI}

\DeclareMathOperator{\pf}{pf}

\DeclareMathOperator{\res}{res}
\newcommand{\Gr}{\mathbf{Gr}}

\newcommand{\fm}{\mathfrak{m}}

\newcommand{\surj}{\twoheadrightarrow}

\title{Symmetric quivers, invariant theory, and saturation theorems
  for the classical groups} 
\author{Steven V Sam} 
\date{October 20, 2011}

\begin{document}

\setcounter{secnumdepth}{5}
\setcounter{tocdepth}{2}

\maketitle

\begin{abstract} Let $G$ denote either a special orthogonal group or a
  symplectic group defined over the complex numbers. We prove the
  following saturation result for $G$: given dominant weights
  $\lambda^1, \dots, \lambda^r$ such that the tensor product
  $V_{N\lambda^1} \otimes \cdots \otimes V_{N\lambda^r}$ contains
  nonzero $G$-invariants for some $N \ge 1$, we show that the tensor
  product $V_{2\lambda^1} \otimes \cdots \otimes V_{2\lambda^r}$ also
  contains nonzero $G$-invariants. This extends results of
  Kapovich--Millson and Belkale--Kumar and complements similar results
  for the general linear group due to Knutson--Tao and
  Derksen--Weyman. Our techniques involve the invariant theory of
  quivers equipped with an involution and the generic representation
  theory of certain quivers with relations.
\end{abstract}


\section{Introduction.}

Throughout, we fix an algebraically closed field $K$. We shall assume
that $K$ is of characteristic 0 in the introduction. However, as some
results in this paper extend to positive characteristic, we will
mention what assumptions we make on the characteristic within each
section of the paper.

When $G$ is a reductive group defined over $K$, and $\lambda$ is a
dominant weight of $G$, the notation $V_\lambda$ denotes an
irreducible representation of $G$ with highest weight $\lambda$. Also,
if $W$ is a representation of $G$, we use $W^G$ to denote the subspace
of $G$-invariants. The following theorem is the main result of the
paper.

\begin{theorem} \label{bcdsaturation} Let $G$ be either a special
  orthogonal or symplectic group, and let $\lambda^1, \dots,
  \lambda^r$ be dominant weights of $G$. If $(V_{N\lambda^1} \otimes
  \cdots \otimes V_{N\lambda^r})^G \ne 0$ for some $N \ge 1$, then
  $(V_{2\lambda^1} \otimes \cdots \otimes V_{2\lambda^r})^G \ne 0$.
\end{theorem}

We say that 2 is a {\bf saturation factor} for the special orthogonal
and symplectic groups. We will use the notation $\SO(m)$, ${\bf
  O}(m)$, and $\Sp(2n)$ to mean special orthogonal, orthogonal, and
symplectic groups, respectively.

\begin{corollary} \label{cor:spin} Let $G$ be the spin group ${\bf
    Spin}(m)$, and let $\lambda^1, \dots, \lambda^r$ be dominant
  weights of $G$. If $(V_{N\lambda^1} \otimes \cdots \otimes
  V_{N\lambda^r})^G \ne 0$ for some $N \ge 1$, then $(V_{4\lambda^1}
  \otimes \cdots \otimes V_{4\lambda^r})^G \ne 0$.
\end{corollary}

\begin{proof} If $\lambda$ is a dominant weight of ${\bf Spin}(m)$,
  then $2\lambda$ is a dominant weight of $\SO(m)$, and the action of
  ${\bf Spin}(m)$ factors through $\SO(m)$ on each $V_{4\lambda^i}$.
\end{proof}

\subsection{History and related results.}

Before we give an outline for the proof of
Theorem~\ref{bcdsaturation}, we mention some historical context for
the theorem and some results that have previously been proven in this
direction.

The results start with the so-called saturation conjecture proven by
Knutson--Tao \cite{knutson} and Derksen--Weyman \cite{saturation}.

\begin{thm}[Knutson--Tao, Derksen--Weyman] Let $\lambda^1, \dots,
  \lambda^r$ be dominant weights of $G = \GL(n)$. If $(V_{N\lambda^1}
  \otimes \cdots \otimes V_{N\lambda^r})^G \ne 0$, then
  $(V_{\lambda^1} \otimes \cdots \otimes V_{\lambda^r})^G \ne 0$.
\end{thm}

This problem itself was inspired by Klyachko's solution
\cite{klyachko} of Horn's problem of characterizing the possible
eigenvalues of Hermitian matrices $A_1, \dots, A_r$ whose sum is 0. We
leave the details out and refer to Fulton's paper \cite{fulton} for a
survey and further references.

When $r=3$, this theorem can be restated in terms of the
Littlewood--Richardson rule (see \cite[Theorem 2.3.4]{weyman} or
\cite[\S 12.5]{procesi}), which gives an explicit combinatorial recipe
for calculating the dimension of the $G$-invariant subspace of a
triple tensor product, or equivalently, for calculating tensor product
multiplicities. However, the formulation of this rule is not conducive
to proving the saturation property. The proof of Knutson--Tao involves
formulating a new combinatorial rule which more manifestly possesses
the saturation property. However, this approach seems to be difficult
to generalize. Our paper will follow the ideas of
Derksen--Weyman. Before reviewing the ideas from that paper, we
mention some other saturation results to put
Theorem~\ref{bcdsaturation} into perspective. We refer the reader to
\cite{kumar} for more results and conjectures related to tensor
product multiplicities.

\begin{thm}[Kapovich--Millson] Let $G$ be a simple connected group
  over $K$, and let $\lambda^1, \dots, \lambda^r$ be dominant weights
  of $G$ such that $\lambda^1 + \cdots + \lambda^r$ is in the root
  lattice of $G$. Let $k$ be the least common multiple of the
  coefficients of the highest root of $G$ written in terms of simple
  roots. If $(V_{N\lambda^1} \otimes \cdots\otimes V_{N\lambda^r})^G
  \ne 0$ for some $N \ge 1$, then $(V_{k^2\lambda^1} \otimes \cdots
  \otimes V_{k^2\lambda^r})^G \ne 0$. 
\end{thm}

See \cite[Corollary 7.3 and Remark 7.2]{kapovich}. For the special
orthogonal and symplectic groups, this gives a saturation factor of 4,
which our Theorem~\ref{bcdsaturation} improves to 2 (and drops the
assumption that $\lambda^1 + \cdots + \lambda^r$ be in the root
lattice). The improvements for the odd orthogonal groups and
symplectic groups have previously been shown by Belkale and Kumar
\cite[Theorems 6 and 7]{evensaturation}. So
Theorem~\ref{bcdsaturation} provides a new result for $G =
\SO(2n)$. Even in the known cases of the odd orthogonal groups and
symplectic groups, we believe that our proof still has merit in that
the ideas are uniform with respect to the classical groups and in some
sense are more elementary. Corollary~\ref{cor:spin} provides a slight
improvement to the general result of Kapovich--Millson as the next
example shows.

\begin{example}
  Let $G = {\bf Spin}(10)$, label the Dynkin diagram ${\rm D}_5$ as
  \[
  \xymatrix @-1.2pc { & & 4 & \\ 1 \ar@{-}[r] & 2 \ar@{-}[r] & 3 \ar@{-}[u]
    \ar@{-}[r] & 5}
  \]
  and let $\omega_1, \dots, \omega_5$ be the corresponding fundamental
  weights. Then $V_{4(\omega_2 + \omega_4 + \omega_5)}^{\otimes 2}
  \otimes V_{4(2\omega_1 + \omega_2 + \omega_5)}$ contains a nonzero
  $G$-invariant, but $2(\omega_2 + \omega_4 + \omega_5) + (2\omega_1 +
  \omega_2 + \omega_5)$ is not in the root lattice. Furthermore, none
  of these weights are sums of minuscule weights.
\end{example}

The relevance of the condition that $\lambda^1 + \cdots + \lambda^r$
be in the root lattice is that $(V_{\lambda^1} \otimes \cdots \otimes
V_{\lambda^r})^G$ can only be nonzero if this condition holds:
$\lambda^1 + \cdots + \lambda^r$ and 0 would be weights of the tensor
product, and any weights of a representation are equal modulo the root
lattice. Conjecturally, the saturation constant in the general result
of Kapovich and Millson for the even spin group can be shown to be
1. The index of the root lattice in the weight lattice for
$\SO(2n+1)$, $\Sp(2n)$, and $\SO(2n)$, is 2, 1, and 2, respectively,
so this conjecture includes the statement of
Theorem~\ref{bcdsaturation}. This more general statement has been
proven for ${\bf Spin}(8)$, see \cite{spin8}.

\begin{example} Theorem~\ref{bcdsaturation} cannot be strengthened by
  replacing even saturation with actual saturation because the
  condition that $\lambda^1 + \cdots + \lambda^r$ be in the root
  lattice is not a linear condition. More explicitly, we have the
  following tensor product decompositions in types B, C, and D:
  \begin{align*}
    V_{(1)} \otimes V_{(1)} &= V_{(2)} \oplus V_{(1,1)} \oplus
    V_{(0)}\\
    V_{(2)} \otimes V_{(2)} &= V_{(4)} \oplus V_{(3,1)} \oplus
    V_{(2,2)} \oplus V_{(2)} \oplus V_{(1,1)} \oplus V_{(0)},
  \end{align*}
  where we have identified weights with partitions as in
  Section~\ref{sec:classicalrepn}. In other words, $(V_{(1)}^{\otimes
    3})^G = 0$ but $(V_{(2)}^{\otimes 3})^G \ne 0$. Note that the
  weight $(3)$ is in the root lattice in type B, but not in types C
  and D, so even taking into account the root lattice condition, the
  saturation factor cannot be 1 for type B. Counterexamples are also
  known for type C.
\end{example}

\subsection{Outline of the paper.}

We first outline the proof of saturation for the general linear group
due to Derksen--Weyman \cite{saturation}. That paper is a study of the
semi-invariants of quivers without oriented cycles, see
Section~\ref{section:quiversintro} for definitions. The connection to
tensor product multiplicities is as follows. Given dominant weights
$\lambda^1, \dots, \lambda^r$ for $\GL(n)$, one can produce a quiver
$Q$, a dimension vector $\beta$, and a weight $\theta$ such that
\begin{align} \label{eqn:semiinvariants}
(V_{N\lambda^1} \otimes \cdots \otimes V_{N\lambda^r})^{\GL(n)} \cong
\SI(Q, \beta)_{N\theta}
\end{align}
for any $N \ge 1$. So the saturation problem can be reduced to proving
a related saturation problem for semi-invariants of the quiver $Q$. In
fact, they prove this saturation result for any quiver without
oriented cycles. More precisely, one proves the following
equivalences:
\begin{align} \label{eqn:equivalences}
\SI(Q, \beta)_{N\theta} \ne 0 \iff \ext^1(N\alpha, \beta) = 0 \iff
\ext^1(\alpha, \beta) = 0 \iff \SI(Q, \beta)_\theta \ne 0,
\end{align}
see Section~\ref{section:semicontinuity} for definitions.

The benefit from working in this more general context is that it
allows reductions to smaller quivers or smaller dimension vectors that
are not necessarily related to tensor product multiplicities. There
are actually two key inductions. Given a representation $W$ of $Q$,
Schofield \cite{schofield} introduced the determinantal semi-invariant
$c^W$ which is a nonzero function if and only if $\ext^1(W, \beta)$
generically vanishes. The main result of \cite{saturation} is that the
spaces of semi-invariants are linearly spanned by the $c^W$ for
various representations $W$ of dimension $\alpha$. This provides the
first and third equivalences above. Their proof involves a series of
reductions to smaller quivers with the base case being the generalized
Kronecker quiver on two vertices. The second equivalence was provided
by Schofield \cite{generalreps}, who showed that the dimension of the
generic extension group $\ext^1(\alpha, \beta)$ can be calculated
recursively from smaller dimension vectors, and the form of this
recursion shows that the dimension is 0 if and only if $\alpha$
satisfies a finite system of linear inequalities that depend only on
$\beta$ and $\theta$.

~

So one may hope that the ideas in the above proof can be generalized
to the other classical groups. In trying to get an analogue of
\eqref{eqn:semiinvariants} for the orthogonal and symplectic groups,
one needs to introduce two complications to the quiver $Q$. First, one
introduces an involution on $Q$ and restricts to studying the
representations compatible with this involution. We call these
symmetric quivers and their symmetric representations. Second, one
replaces the path algebra $KQ$ by a certain quotient ring
$KQ/I$. Geometrically, both of these complications amount to
restricting to certain subvarieties of the representation varieties of
$Q$.

With regard to the first complication, the spaces of semi-invariants
are no longer spanned by the determinantal semi-invariants. The fact
responsible for this is that the determinant of a generic
skew-symmetric matrix is the square of its Pfaffian. Motivated by
this, we introduce Pfaffian semi-invariants as square roots of
determinantal semi-invariants and show that they linearly span the
space of semi-invariants for symmetric quivers. This is the content of
Section~\ref{section:pfaffians}. We remark that we know of no general
criterion for a determinantal semi-invariant to possess a square root.

The second complication has the following effect. Any submodule of a
projective module over $KQ$ is also projective. This is the same as
saying that the global dimension of $KQ$ is at most 1. This fact was
used extensively in Schofield's proofs. In general, the global
dimension of $KQ/I$ is bigger than 1. In our case, it is 2, so the
problem is not so bad, but Schofield's results no longer apply. To get
around this, we extend Schofield's results in
Section~\ref{section:relations} in the case of global dimension 2
under certain assumptions which are sufficient for our applications.

Finally, in Section~\ref{section:bcdLRcoeff}, we combine these two
generalizations to prove the analogue of \eqref{eqn:equivalences}.

\subsection*{Conventions.}

All topological notions refer to the Zariski topology. For us, a
variety is a separated finite type scheme over $K$ which need not be
irreducible nor reduced. Any field that is implicitly used in this
paper refers to the algebraically closed field $K$. The set of
integers is denoted by $\Z$, and the set of nonnegative integers is
denoted by $\N$.

\subsection*{Acknowledgements.} 

The author thanks Jerzy Weyman for suggesting this problem and for
helpful discussions. We also thank Artem Lopatin for pointing out the
article \cite{lopatin} which allows us to extend
Lemma~\ref{fund:tensors} to positive characteristic. The author was
supported by an NSF graduate research fellowship and an NDSEG
fellowship while this work was done. The calculations for some of the
examples involving tensor product multiplicities were done using the
software LiE \cite{lie}.

\section{Semi-invariants of symmetric
  quivers.} \label{section:pfaffians} 

\subsection{Quivers.} \label{section:quiversintro}

As a general reference for quivers, we refer to \cite{assem}. In this
section, no assumption on the characteristic of $K$ is made. A {\bf
  quiver} $Q$ is the data $(Q_0, Q_1, t, h)$ where $Q_0$ is the {\bf
  vertex set}, $Q_1$ is the {\bf arrow set}, and $t, h \colon Q_1 \to
Q_0$ are functions.  For $a \in Q_1$ we call $ta$ and $ha$ its {\bf
  tail} and {\bf head} and depict it by the following diagram:
\[
ta \xrightarrow{a} ha,
\]
so that the definitions of paths, cycles, etc. should be
self-evident. The {\bf path algebra} $KQ$ is defined as a vector space
to be the finite linear combinations of paths in $Q$. The product of
two paths $p_1$ and $p_2$ is defined to be the concatenation $p_1p_2$
if this is a well-defined path (the sequence $a_n \cdots a_1$ means
the path that starts with $a_1$ and ends with $a_n$) and is 0
otherwise.

We will always assume that $Q_0$ is finite and that $Q$ has no
directed cycles.

\subsubsection{Representations.}

Elements $\beta \in \N^{Q_0}$ are {\bf dimension vectors}. The {\bf
  Euler form} $\langle,\rangle_Q \colon \Z^{Q_0} \times \Z^{Q_0} \to
\Z$ is
\[
\langle \alpha, \beta \rangle_Q = \sum_{x \in Q_0} \alpha(x) \beta(x)
- \sum_{a \in Q_1} \alpha(ta) \beta(ha).
\]
When $Q$ is clear from context, we will drop the subscript. We define
the {\bf representation variety}
\[
\Rep(Q,\beta) = \bigoplus_{a \in Q_1} \hom(K^{\beta(ta)},
K^{\beta(ha)}),
\]
and the groups
\begin{align*}
\GL(Q,\beta) = \prod_{x \in Q_0} \GL(K^{\beta(x)}), \quad
\SL(Q,\beta) = \prod_{x \in Q_0} \SL(K^{\beta(x)})
\end{align*}
which act on $\Rep(Q,\beta)$ via 
\[
(g_x)_{x \in Q_0} \cdot (\phi_a)_{a \in Q_1} = ( g_{ha} \phi_a
g_{ta}^{-1} )_{a \in Q_1}.
\]

A representation of $Q$ of dimension $\beta$ is the assignment of a
vector space $V(x)$ of dimension $\beta(x)$ for each $x \in Q_0$, as
well as a linear map $V_a \colon V(ta) \to V(ha)$ for each $a \in
Q_1$. In this case, we write $\dim V = \beta$. A morphism of two
representations $\phi \colon V \to W$ is a collection of linear maps
$(\phi_x)_{x \in Q_0}$ such that the evident squares all
commute. Geometrically, representations of $Q$ correspond to
$K$-points in $\Rep(Q, \beta)$, and two representations are isomorphic
if and only if they belong to the same $\GL(Q,
\beta)$-orbit. Algebraically, representations of $Q$ are the same as
(left) modules of the path algebra $KQ$. This latter definition makes
it clear how we can define extensions, projective resolutions, etc.

\subsubsection{Semi-invariants.} \label{section:quiverSI}

For an affine variety $X$, we denote its coordinate ring by
$K[X]$. For a quiver $Q$ and dimension vector $\beta \in \N^{Q_0}$, we
define the ring of {\bf semi-invariants} as the invariants
\[
\SI(Q, \beta) = K[\Rep(Q, \beta)]^{\SL(Q, \beta)}.
\]
We grade it by the characters $\chi$ of $\GL(Q, \beta)$
\[
\SI(Q, \beta)_\chi = \{ f \in \SI(Q, \beta) \mid g \cdot f = \chi(g)f
\text{ for all } g \in \GL(Q, \beta) \}.
\]
We call $\chi$ the {\bf weight} of these semi-invariants. The
characters of $\GL(Q, \beta)$ are of the form
\[
(g_x)_{x \in Q_0} \mapsto \prod_{x \in Q_0} (\det g_x)^{\sigma(x)}
\]
for $\sigma \in \Z^{Q_0}$. Hence we can identify weights with elements
of $\Z^{Q_0}$.

~

For $x,y \in Q_0$, let $[x,y]$ denote the $K$-vector space whose basis
is the paths from $x$ to $y$. The indecomposable projective
representations of $Q$ are indexed by $Q_0$: for $x \in Q_0$, set
$P_x$ to be the representation with $P_x(y) = [x,y]$ for $y \in Q_0$
and $P_{x,a} \colon [x,ta] \to [x,ha]$ is the natural map which
appends the arrow $a$ to the end of a path from $x$ to $ta$. There is
a {\bf canonical resolution} for each representation $V$ of $Q$
\[
0 \to \bigoplus_{a \in Q_1} V(ta) \otimes P_{ha} \xrightarrow{d^V}
\bigoplus_{x \in Q_0} V(x) \otimes P_x \to V \to 0,
\]
where the differential $d^V$ is described as follows. Given $v \otimes
p \in V(ta) \otimes P_{ha}$, send it to $(V(ta) \otimes P_{ta}) \oplus
(V(ha) \otimes P_{ha})$, where the map to the first factor is induced
by the inclusion $P_{ha} \subset P_{ta}$ given by appending $a$ to the
beginning of a path, and the map to the second factor is $-V_a \otimes
1_{P_{ha}}$.

Given another representation $W$, we define $d^V_W = \hom(d^V, W)$. We
can also define $d^V_W$ by
\begin{align*}
  \bigoplus_{x \in Q_0} \hom(V(x), W(x)) &\xrightarrow{d^V_W}
  \bigoplus_{a \in Q_1} \hom(V(ta), W(ha))\\
  (\phi_x)_{x \in Q_0} &\mapsto (\phi_{ha}V_a - W_a\phi_{ta})_{a \in Q_1}
\end{align*}

Let $\alpha = \dim V$ and $\beta = \dim W$. In the case that $\langle
\alpha, \beta \rangle = 0$, $d^V_W$ is a map between vector spaces of
the same dimension. By fixing bases, we can define its determinant
$c^V_W = \det d^V_W$. This is only well-defined up to a nonzero scalar
multiple, but this choice will not be important for us. Then $c^V$
gives a polynomial function on $\Rep(Q, \beta)$, which is a
semi-invariant of weight $\sigma^\circ_\alpha$ defined by
\begin{align} \label{usualweight}
\sigma^\circ_\alpha(x) = \langle \alpha, \eps_x \rangle
\end{align}
where $\eps_x$ is the vector defined by $\eps_x(x) = 1$ and $\eps_x(y)
= 0$ for $y \ne x$. These semi-invariants are called {\bf
  determinantal semi-invariant}. See \cite{schofield} for some basic
properties. We recall a fundamental result on semi-invariants of
quivers due to Derksen and Weyman \cite[Theorem 1]{saturation}.

\begin{thm}[Derksen--Weyman] Let $Q$ be a quiver without oriented
  cycles and let $\beta \in \N^{Q_0}$. Then $\SI(Q, \beta)$ is
  linearly spanned by determinantal semi-invariants $c^V$ where
  $\langle \dim V, \beta \rangle = 0$.
\end{thm}

\subsection{Symmetric quivers.} 

Suppose that the characteristic is different from 2. 

Let $\tau$ denote an involution $\tau \colon Q_0 \to Q_0$ and $\tau
\colon Q_1 \to Q_1$ such that $t\tau(a) = \tau(ha)$ and $h\tau(a) =
\tau(ta)$. We pick a sign function defined on the $\tau$-fixed
vertices and arrows $s \colon Q_0^\tau \cup Q_1^\tau \to \{\pm
1\}$. The data $(Q_0, Q_1, \tau, s)$ is a {\bf symmetric quiver}. If
$Q$ is a symmetric quiver, then $Q^\circ$ denotes the underlying
quiver. Our definition of symmetric quiver is called a {\bf signed
  quiver} in \cite{signed}. For $i=0,1$, we can partition
\[
Q_i = Q_i^+ \cup Q_i^\tau \cup Q_i^-,
\]
such that $Q_i^\tau$ is the fixed point set of $\tau$ and $Q_i^- =
\tau(Q_i^+)$.
Also, we set $Q_0^{\tau_{\pm}}$ to be the set of $x \in Q_0^\tau$ such
that $s(x) = \pm 1$, so $Q_0^\tau = Q_0^{\tau_+} \cup
Q_0^{\tau_-}$. 

\subsubsection{Representations.}

A dimension vector $\beta \in \N^{Q_0}$ is {\bf symmetric} if
$\beta(\tau(x)) = \beta(x)$ for all $x \in Q_0$ and $\beta(x)$ is even
whenever $x \in Q_0^{\tau_+}$. Given a symmetric dimension vector
$\beta$, a {\bf symmetric representation} $V$ of dimension $\beta$ is
defined by assigning a vector space $V(x)$ of dimension $\beta(x)$ to
each $x \in Q_0^+ \cup Q_0^\tau$. For $x \in Q_0^+$, we assign the
dual vector space $V(\tau(x)) = V(x)^*$ to $\tau(x)$. For $x \in
Q_0^{\tau_+}$, we endow $V(x)$ with a nondegenerate symmetric bilinear
form, and for $x \in Q_0^{\tau_-}$, we endow $V(x)$ with a
nondegenerate skew-symmetric bilinear form. In particular, the form
gives an isomorphism $J_x \colon V(x) \to V(x)^*$ that we fix. Note
that $J_x^{-1} = J_x^* = \eps J_x$ where $x \in Q_0^{\tau_\eps}$. For
each $a \in Q_1$, we assign a linear map $V_a \colon V(ta) \to V(ha)$
such that
\begin{compactenum}
\item If $a \in Q_1^{\tau_\eps}$, then $V_a = \eps V_a^*$ where we are
  identifying $V(ta)$ and $V(ta)^{**}$ in the canonical way.
\item If $a \in Q_1^+$ and $ta, ha \notin Q_0^\tau$, then $V_a =
  V_{\tau(a)}^*$.
\item If $a \in Q_1^+$ and $ha \in Q_0^\tau$, then $V_{\tau(a)} =
  V_a^* J_{ha}$, and similarly if instead $ta \in Q_0^\tau$.
\end{compactenum}

~

We define the {\bf symmetric representation variety}
\begin{align*}
  \SRep(Q, \beta) &= \bigoplus_{a \in Q_1^+} \hom(K^{\beta(ta)},
  K^{\beta(ha)}) \oplus \bigoplus_{\substack{a \in Q_1^{\tau_+} }}
  \Sym^2(K^{\beta(ta)})^* \oplus \bigoplus_{\substack{a \in
      Q_1^{\tau_-}}} \bigwedge^2 (K^{\beta(ta)})^*
\end{align*}
and the corresponding groups
\begin{align*} 
  {\bf G}(Q, \beta) &= \prod_{x \in Q_0^+} \GL(K^{\beta(x)}) \times
  \prod_{\substack{x \in Q_0^\tau \\ s(x) = 1}} {\bf O}(K^{\beta(x)})
  \times \prod_{\substack{x \in Q_0^\tau \\ s(x) = -1}} \Sp(K^{\beta(x)}) \\
  \SG(Q, \beta) &= \prod_{x \in Q_0^+} \SL(K^{\beta(x)}) \times
  \prod_{\substack{x \in Q_0^\tau \\ s(x) = 1}} \SO(K^{\beta(x)})
  \times \prod_{\substack{x \in Q_0^\tau \\ s(x) = -1}}
  \Sp(K^{\beta(x)}).
\end{align*}
These groups act on the symmetric representation variety as follows
(here we interpret $g_{\tau(x)} = (g_x^{-1})^t$ where the transpose is
with respect to the natural bilinear form on $V(x) \oplus V(\tau(x))$
for $x \in Q_0^+$):
\[
(g_x)_{x \in Q_0^+ \cup Q_0^\tau} \cdot (\phi_a)_{a \in Q_1^+ \cup
  Q_1^\tau} = ( g_{ha} \phi_a g_{ta}^{-1} )_{a \in Q_1^+ \cup
  Q_1^\tau}. 
\]
We have a natural identification ${\bf G}(Q, \beta) \subset
\GL(Q^\circ, \beta)$ such that $\SRep(Q, \beta)$ is a ${\bf G}(Q,
\beta)$-invariant subvariety of $\Rep(Q^\circ, \beta)$. The 
symmetric representations of $Q$ correspond to the $K$-points of
$\SRep(Q, \beta)$, and isomorphism of two symmetric representations is
defined as being in the same ${\bf G}(Q, \beta)$-orbit.

\subsubsection{Semi-invariants.} 

As in Section~\ref{section:quiverSI}, when $Q$ is a symmetric quiver
and $\beta$ is a symmetric dimension vector, we define the ring of
{\bf symmetric semi-invariants}
\begin{align*}
  \SSI(Q, \beta) &= K[\SRep(Q, \beta)]^{\SG(Q, \beta)},
\end{align*}
which has a grading by characters of ${\bf G}(Q,\beta)$. The
determinant is the only nontrivial character of the orthogonal group
and has order 2, while the symplectic group has no nontrivial
characters. So we can identify weights with elements of $\Z^{Q_0^+}
\times (\Z/2)^{Q_0^{\tau_+}}$. Now we describe the class of
semi-invariants that we study in this article.

Let $V$ be a representation of $Q^\circ$ of dimension $\alpha$. Recall
the definition of the polynomial function $c^V$ on $\Rep(Q^\circ,
\beta)$ given in Section~\ref{section:quiverSI}. We are interested in
the restriction of $c^V$ to $\SRep(Q, \beta)$. The weight
$\sigma_\alpha \in \Z^{Q_0^+} \times (\Z/2)^{Q_0^{\tau_+}}$ is defined
by
\begin{align} \label{weight}
  \sigma_\alpha(x) = 
  \begin{cases} \langle \alpha, \eps_x - \eps_{\tau(x)}
    \rangle & \text{if } x \in Q_0^+ \\
    \langle \alpha, \eps_x \rangle & \text{if } x \in
    Q_0^{\tau_+} \end{cases}.
\end{align}

\begin{remark} \label{remark:integerweights} While $\langle \alpha,
  \eps_x \rangle$ is an element of $\Z/2$ when $x \in Q_0^{\tau_+}$,
  it will be convenient for us later to think of it as an integer.
\end{remark}

If the restriction of the polynomial function $c^V$ to $\SRep(Q,
\beta)$ is the square of another polynomial function, we define
$\pf^V$ to be a square root of this function. Since $\SI(Q, \beta)$ is
a domain, this is well-defined up to a choice of nonzero scalar since
$x^2 = y^2$ implies that $x = \pm y$. Since $\SG(Q, \beta)$ has no
nontrivial characters, $\pf^V$ is also a semi-invariant. We call this
a {\bf Pfaffian semi-invariant}. The weight of $\pf^V$ is
$\frac{1}{2}\sigma_\alpha$: this makes sense for $x \in Q_0^+$; for $x
\in Q_0^{\tau_+}$, we interpret $\frac{1}{2} \sigma_\alpha(x)$, which
must be an integer, to be the residue of $\frac{1}{2} \langle \alpha,
\eps_x \rangle$ in $\Z/2$. Since $c^{V \oplus V} = (c^V)^2$, one can
always interpret a determinantal semi-invariant as a Pfaffian
semi-invariant: $c^V = \pf^{V \oplus V}$.


\begin{theorem} \label{pfaffians} Assume that the characteristic of
  $K$ is different from $2$. Let $Q$ be a symmetric quiver without
  oriented cycles and let $\beta$ be a symmetric dimension vector for
  $Q$. The space of symmetric semi-invariants $\SSI(Q, \beta)$ is
  linearly spanned by the Pfaffian semi-invariants $\pf^V$ such that
  $\langle \dim V, \beta \rangle = 0$. The weight of $\pf^V$ is
  $\frac{1}{2}\sigma_{\dim V}$.
\end{theorem}

We remark that Theorem~\ref{pfaffians} was proven in the case of
finite type and tame quivers without oriented cycles in the paper
\cite{aragona} along with a determination for when the determinantal
semi-invariants admit square roots. The technique of proof in that
paper involves extending the idea of reflection functors, but our
approach will follow the ideas in \cite{saturation} closely. Hence
there will be three steps. The first step is to reduce to the case of
a symmetric quiver that has a unique sink and source and such that the
weight is 1 at the sink, and 0 elsewhere. The second step is to show
how one can remove vertices of weight 0. The third and last step is to
handle the case of a symmetric quiver with two vertices. The proof
will be given in Section~\ref{sec:proof} after we state some necessary
background from invariant theory.

\subsection{Some results from invariant
  theory.} \label{invariantsection}

\subsubsection{Fundamental invariants.}

We recall the first fundamental theorems of invariant theory for the
classical groups.

\begin{theorem} \label{theorem:fund} Let $V$ be a vector space of
  dimension $n$ over a field $K$.
  \begin{compactenum}[\rm (a)]
  \item \label{fund:general} We have
    \begin{align*} 
      K[V^{\oplus p} \oplus (V^*)^{\oplus q}]^{\GL(V)} = K[u_{i,j}],
    \end{align*}
    where $u_{i,j}$ $(1 \le i \le p$, $1 \le j \le q)$ is the function
    defined by $(v, f) \mapsto f_j(v_i)$. 

  \item \label{fund:symplectic} Now assume that $V$ has a
    skew-symmetric nondegenerate bilinear form $(,)$. Then
    \begin{align*} 
      K[V^{\oplus p}]^{\Sp(V)} = K[u_{i,j}]
    \end{align*}
    where $u_{i,j}$ $(1 \le i, j \le p)$ is the function defined by
    $(v_1, \dots, v_p) \mapsto (v_i, v_j)$. 

  \item \label{fund:orthogonal} Now assume that the characteristic of
    $K$ is different from $2$ and that $V$ has a symmetric
    nondegenerate bilinear form $(,)$. Then
    \begin{align*} 
      K[V^{\oplus p}]^{{\bf O}(V)} = K[u_{i,j}]
    \end{align*}
    where $u_{i,j}$ $(1 \le i, j \le p)$ is the function defined by
    $(v_1, \dots, v_p) \mapsto (v_i, v_j)$. 

  \end{compactenum}
\end{theorem}

\begin{proof} For \eqref{fund:general}, \eqref{fund:symplectic}, and
  \eqref{fund:orthogonal}, respectively, see \cite[\S
  13.6.3]{procesi}, \cite[\S 13.8.5, Theorem 1]{procesi}, and \cite[\S
  13.8.5, Theorem 2]{procesi}.
\end{proof}

In the quiver context, \eqref{fund:general} implies the following
fact. Given a representation variety of a quiver of the form $K^p
\xrightarrow{A} K^n \xrightarrow{B} K^q$, we can interpret $A$ as $p$
vectors of $K^n$ and $B$ as $q$ covectors. Hence the
$\GL(K^n)$-invariants of $K[A,B]$ is just $K[AB]$, the subring
generated by the entries of the product $AB$. Both
\eqref{fund:orthogonal} and \eqref{fund:symplectic} have similar
interpretations for symmetric representations that we will be using
later.

We will also require the following extension of
Theorem~\ref{theorem:fund}\eqref{fund:general}.

\begin{lemma} \label{fund:tensors} Let $V$ be a vector space over a
  field $K$ of characteristic different from $2$. Define
  \[
  W_{p,q,r,s} = V^{\oplus p} \oplus (V^*)^{\oplus q} \oplus
  (\bigwedge^2 V^*)^{\oplus r} \oplus (\Sym^2 V^*)^{\oplus s}.
  \]
  Then
  \begin{align*} 
    K[W_{p,q,r,s}]^{\GL(V)} = K[u_{i,j}, \eps_{i,j,k}, \eta_{i,j,k}],
  \end{align*}
  where 
  \begin{compactitem}
  \item $\eps_{i,j,k}$ $(1 \le i < j \le p$, $1 \le k \le r)$ is the
    function defined by $(v, f, \xi, \zeta) \mapsto \xi_k(v_i,
    v_j)$,
  \item $\eta_{i,j,k}$ $(1 \le i \le j \le p$, $1 \le k \le s)$ is the
    function defined by $(v, f, \xi, \zeta) \mapsto \zeta_k(v_i,
    v_j)$, and
  \item $u_{i,j}$ $(1 \le i \le p$, $1 \le j \le q)$ is the function
    defined by $(v, f, \xi, \zeta) \mapsto f_j(v_i)$.
  \end{compactitem}
\end{lemma}

\begin{proof}
  This follows from \cite[Theorem 1]{lopatin}, but we will give a
  simpler proof using classical invariant theory in the case when the
  characteristic of $K$ is 0.
  
  For the definitions of full polarization and restitution, we refer
  to \cite[\S 3.2.2]{procesi}. First we note that the invariants of
  $K[W_{p,q,r,s}]$ are graded by $\N^{p+q+r+s}$. Given a
  multihomogeneous invariant, its full polarization is a multilinear
  invariant function on $W = W_{p',q',r',s'}$ for some $p',q',r',s'$
  depending on the degree of the invariant. We have an inclusion of
  $\GL(V)$-representations $W \subset W' = V^{\oplus p'} \oplus
  (V^*)^{\oplus q'} \oplus (V^* \otimes V^*)^{\oplus (r'+s')}$. The
  space of multilinear invariants of $W'$ is $(V^{\otimes p'} \otimes
  (V^*)^{\otimes (q'+2r'+2s')})^{\GL(V)}$. We know from
  Theorem~\ref{theorem:fund}\eqref{fund:general} that $p' =
  q'+2r'+2s'$ and that the invariants are linear combinations of the
  monomials $u_w = u_{1, w(1)} \cdots u_{p', w(p')}$ where $w$ is in
  the symmetric group $\Sigma_{p'}$. Since we are working in
  characteristic 0 and $\GL(V)$ is reductive, the restrictions of the
  $u_w$ from $W'$ to $W$ generate the multilinear
  $\GL(V)$-invariants. Finally, the restitution of $u_w$ to a
  multihomogeneous invariant on $W_{p,q,r,s}$ is a product of
  $u_{i,j}$, $\eps_{i,j,k}$, $\eta_{i,j,k}$, so our claim is proven
  since every multihomogeneous invariant is the restitution of some
  multilinear invariant. 
\end{proof}

\subsubsection{Schur functors.}

We recall some facts about the representation theory of the general
linear group. For this section, we make no assumptions on the
characteristic of $K$, except for the last two results.

Let $V$ and $W$ be vector spaces, and let $\lambda = (\lambda_1 \ge
\lambda_2 \ge \cdots)$ be a partition of nonnegative integers. The
notation $\ell(\lambda)$ denotes the largest $n$ such that $\lambda_n
\ne 0$, and $\lambda'$ denotes the conjugate partition of
$\lambda$. For $c \in \N$, we define $c\lambda = (c\lambda_1,
c\lambda_2, \dots)$. 

Let $\bS_\lambda(V)$ be the {\bf Schur functor}, which is an
irreducible representation of $\GL(V)$ with highest weight
$\lambda$. (This is denoted $\bL_{\lambda'}(V)$ in \cite[\S
2.1]{weyman}.) Given two representations $U$ and $U'$, we will write
$U \sim U'$ if each representation possesses a filtration such that
the associated graded representations are isomorphic. If $K$ has
characteristic 0, $\sim$ is the same as isomorphism. 

\begin{proposition} We have the following decompositions:
\begin{align}
  \label{eqn:sym2plethysm} \DS \Sym^m(\Sym^2 V) &\sim
  \bigoplus_{|\lambda| = m} \bS_{2\lambda}(V),\\
  \label{eqn:alt2plethysm} \DS \Sym^m(\bigwedge^2 V) &\sim
  \bigoplus_{|\lambda| = m} \bS_{(2\lambda)'}(V),\\
  \label{eqn:cauchy} \DS \Sym^m(V \otimes W) &\sim
  \bigoplus_{|\lambda| = m} \bS_\lambda(V) \otimes \bS_\lambda(W).
\end{align}
where the first two identities are $\GL(V)$-equivariant, and the last
one is $\GL(V) \times \GL(W)$-equivariant.
\end{proposition}

For \eqref{eqn:sym2plethysm} and \eqref{eqn:alt2plethysm}, see
\cite[Theorem 11.4.5]{procesi} or \cite[Proposition 2.3.8]{weyman} and
for \eqref{eqn:cauchy}, see \cite[Chapter 9, (6.3.2)]{procesi} or
\cite[Corollary 2.3.3]{weyman}.

\begin{corollary} \label{cor:wedge2invariants} Suppose that $\dim V =
  2n$ and $\lambda$ is a partition of $n$. Then
  $\bS_\lambda(\bigwedge^2 V)$ contains $\SL(V)$-invariants if and
  only if $\lambda = (n)$.
\end{corollary}

\begin{proof} By \cite[Theorem 3.7]{boffi}, the tensor product of two
  Schur functors is $\sim$-equivalent to a direct sum of Schur
  functors, and these multiplicities are independent of
  characteristic. So $\bS_\lambda(\bigwedge^2 V)$ is a subquotient of
  $(\bigwedge^2 V)^{\otimes n}$. The $\SL(V)$-invariants of this
  tensor product are the submodules isomorphic to $\bigwedge^{2n}
  V$. By Pieri's rule (see \cite[\S 9.10.2]{procesi} or
  \cite[Corollary 2.3.5]{weyman}), this representation appears with
  multiplicity 1, and this comes from $\Sym^n(\bigwedge^2 V)$ by
  \eqref{eqn:alt2plethysm}.
\end{proof}

\begin{proposition} \label{prop:dualschur} Assume that the
  characteristic of $K$ is $0$. If $\dim V = n$, then we have an
  isomorphism of $\GL(V)$-representations $\bS_\lambda(V^*) =
  (\bigwedge^n V^*)^{\otimes \lambda_1} \otimes \bS_\mu V$ where
  $\mu_i = \lambda_1 - \lambda_{n+1-i}$.
\end{proposition}

See \cite[\S 2.3.3]{procesi}. 

\begin{proposition} \label{prop:SLinvariants} Assume that the
  characteristic of $K$ is $0$. Write $\lambda' = (1^{d_1}, \dots,
  (n-1)^{d_{n-1}})$ and $\mu' = (1^{e_1}, \dots,
  (n-1)^{e_{n-1}})$. Then $\bS_\lambda(V) \otimes \bS_\mu(V)$ contains
  a nonzero $\SL(V)$-invariant if and only if $d_i = e_{n-i}$ for
  $i=1,\dots,n-1$. In this case, the action of $\GL(V)$ on these
  $\SL(V)$-invariants is via the $m$th power of the determinant, where
  $m = d_1 + \cdots + d_{n-1} = e_1 + \cdots + e_{n-1}$.
\end{proposition}

This follows from the previous proposition and the isomorphism
$\bS_\lambda(V^*)^* \cong \bS_\lambda(V)$ in characteristic 0
\cite[Proposition 2.1.18]{weyman}.

\subsection{Proof of Theorem~\ref{pfaffians}.} \label{sec:proof}

\subsubsection{Reduction to a unique source and sink with nonzero
  weight.}

Let $Q$ be a symmetric quiver with a symmetric dimension vector
$\beta$. It is enough to prove Theorem~\ref{pfaffians} for each weight
space of $\SSI(Q, \beta)$. Given a weight $\sigma \in \Z^{Q^+_0}
\times (\Z/2)^{Q_0^{\tau_+}}$, we lift it to some $\tau$-invariant
weight $\sigma^\circ \in \Z^{Q_0}$.

Form a new symmetric quiver $\ol{Q}$ as follows. We add two new
vertices $x_-$ and $x_+$. If $\sigma^\circ(x) > 0$, then we add
$\sigma^\circ(x)$ arrows $x_- \to x$ and $\sigma^\circ(x)$ arrows
$\tau(x) \to x_+$. If $\sigma^\circ(x) < 0$, then we add
$-\sigma^\circ(x)$ arrows $x \to x_+$ and $-\sigma^\circ(x)$ arrows
$x_- \to \tau(x)$. Let $\ol{\tau}$ be the involution on $\ol{Q}_0$
that switches $x_-$ and $x_+$ and restricts to $\tau$ on $Q_0$. Define
$\ol{\tau}$ on $\ol{Q}_1$ to be the same as $\tau$ on $Q_1$ and to
switch the arrows incident to $x_-$ and $x_+$ (we add arrows two
groups at a time, so just fix an identification of these groups). We
have $\ol{Q}_0^+ = \{x_-\} \cup Q_0^+$, $\ol{Q}_0^\tau = Q_0^\tau$ and
$\ol{Q}_0^- = \{x_+\} \cup Q_0^-$. Also, $\ol{Q}_1^\tau = Q_1^\tau$,
while $\ol{Q}_1^\pm$ is $Q_1^\pm$ plus the arrows incident to
$x_\mp$. Set
\begin{align*}
  \ol{\beta}(x_-) = \ol{\beta}(x_+) = \sum_{x \in Q_0^+} |
  \sigma^\circ(x) | \beta(x),
\end{align*}
and $\ol{\beta}(x) = \beta(x)$ for all $x \in Q_0$. Define a symmetric
weight $\ol{\sigma}$ by $\ol{\sigma}(x_-) = 1$, and $\ol{\sigma}(x) =
0$ for all $x \in Q_0^+ \cup Q_0^\tau$. Also define a $\tau$-invariant
weight $\ol{\sigma}^\circ$ by $\ol{\sigma}^\circ(x_{\pm}) = 1$ and
$\ol{\sigma}^\circ(x) = 0$ for all $x \in Q_0$.

Given a representation $W \in \SRep(\ol{Q}, \ol{\beta})$, let $D(W)$
be the determinant of the matrix formed by taking the direct sum of
all maps incident to $x_-$ (equivalently, all maps incident to
$x_+$). Any function on $\SRep(Q, \beta)$ is naturally a function on
$\SRep(\ol{Q}, \ol{\beta})$. We claim that $c \mapsto Dc$ gives an
isomorphism 
\[
\phi \colon \SSI(Q,\beta)_{\sigma} \to \SSI(\ol{Q},
\ol{\beta})_{\ol{\sigma}}.
\] 
First, suppose that $c \in K[\SRep(Q, \beta)]$ is a symmetric
semi-invariant of weight $\sigma^\circ$. Pick $g \in {\bf G}(\ol{Q},
\ol{\beta})$. Let $W \in \SRep(\ol{Q}, \ol{\beta})$ be a symmetric
representation. Then
\begin{align*}
  (g \cdot D)(W) &= \big( \det g_{x_-} \cdot \prod_{\substack{ x \in
      Q_0^+ \cup Q_0^\tau \\ \sigma^\circ(x) > 0}} (\det
  g_x)^{-\sigma^\circ(x)} \cdot \prod_{\substack{ x \in Q_0^+
      \cup Q_0^\tau \\ \sigma^\circ(x) < 0}} (\det
  g_{\tau(x)})^{\sigma^\circ(x)} \big) D(W),\\ 
  (g \cdot c)(W) &= \big( \prod_{x \in Q_0^+ \cup Q_0^\tau} (\det
  g_x)^{\sigma^\circ(x)} \big) c(W),
\end{align*}
Noting that $g_{\tau(x)} = (g^{-1}_x)^t$ when $x \ne \tau(x)$ and
$(\det g_x)^2 = 1$ when $x = \tau(x)$, we see that
\[
(g \cdot Dc)(W) = (\det g_{x_-}) (Dc)(W),
\]
so $Dc$ is a symmetric semi-invariant of weight $\ol{\sigma}$, and
hence $\phi$ is well-defined. It is clear that $\phi$ is injective,
and surjectivity follows from direct calculations via
\eqref{eqn:cauchy}.

\begin{proposition} Use the notation above. Let $\ol{V}$ be a
  representation of $\ol{Q}$ such that $\langle \ol{V}, \beta
  \rangle_{\ol{Q}} = 0$. If the function $c^{\ol{V}}$ is a square,
  then writing $\pf^{\ol{V}} = Dp$, we have $p = \pf^V$ for some
  representation $V$ of $Q$.
\end{proposition}

\begin{proof} 
  Given $c^{\ol{V}}$, we can extend it to a semi-invariant of
  $\Rep(\ol{Q}, \ol{\beta})$ of weight $\ol{\sigma}^\circ$. By
  \cite[Proposition 2]{saturation}, we can write $c^{\ol{V}} =
  D^-c^VD^+$ for some representation $V$ of $\ol{Q}$ which is
  supported on the subquiver $Q$. The restriction of both $D^+$ and
  $D^-$ to $\SRep(\ol{Q}, \ol{\beta})$ is $D$, so restricting this
  identity to $\SRep(\ol{Q}, \ol{\beta})$, we get $c^{\ol{V}} = D^2
  c^V$. Since $(\pf^{\ol{V}})^2 = c^{\ol{V}}$, we conclude that $p^2 =
  c^V$.
\end{proof}

Therefore, to prove Theorem~\ref{pfaffians}, we may replace $Q$ by
$\ol{Q}$. So we may assume without loss of generality that our
symmetric quiver has a unique source (and hence a unique sink) with
weight 1, and that all other vertices have weight 0.

\subsubsection{Reduction to no weight 0 vertices.} 

To do the second step, we find a vertex $x$ such that $\sigma(x) = 0$
and delete $\{x,\tau(x)\}$ to get a new vertex set $Q'_0$. To define
arrows $Q'_1$, there are two cases to consider, depending on whether
or not $\tau(x) = x$. In both cases, we construct a map of the form
\begin{align} \label{compositionmap} 
  \res^* \colon \SSI(Q', \beta')_{\sigma'} \to \SSI(Q, \beta)_\sigma.
\end{align}
Then we show that it is surjective and that the image of a Pfaffian
semi-invariant is also a Pfaffian semi-invariant so that we can
replace $Q$ by $Q'$.

~

If $x = \tau(x)$, then for any arrows $y \xrightarrow{a} x
\xrightarrow{b} z$, we add an arrow $y \xrightarrow{(a,b)} z$ to
$Q'_1$. We define $\tau'(a,b) = (\tau(b), \tau(a))$. For $(a,b) \in
{Q'_1}^{\tau'}$, we set $s'(a,b) = s(x)$. We define $\tau'$ on $Q_0'$
to be the restriction of $\tau$ from $Q_0$. Similarly, let $\sigma'$
be the restriction of $\sigma$.

Given a symmetric representation $V$ of $Q$ of dimension $\beta$, we
define a symmetric representation $V' = \res V$ of $Q'$ of dimension
$\beta'$, where $\beta'$ is the restriction of $\beta$, by setting
$V_{(a,b)} = V_b V_a$ for all of the new arrows, and by leaving
everything else as is. This gives us a map on symmetric
semi-invariants as in \eqref{compositionmap}. If $s(x) = 1$, we can
consider the map $\bigoplus_{ y \xrightarrow{a} x} K^{\beta(y)} \to
K^{\beta(x)}$ as a choice of $p = \sum_{y \xrightarrow{a} x} \beta(y)$
vectors in $K^{\beta(x)}$. Note that for every $a \in Q_1$, at most
one of $a$ and $\tau(a)$ appears in the sum. Hence
Theorem~\ref{theorem:fund}\eqref{fund:orthogonal} implies that
\eqref{compositionmap} is surjective. Similarly, if $s(x) = -1$, we
can use Theorem~\ref{theorem:fund}\eqref{fund:symplectic} to conclude
that \eqref{compositionmap} is surjective.

~

If $x \ne \tau(x)$, suppose that all arrows between $x$ and $\tau(x)$
are oriented as $x \to \tau(x)$ (which we may assume without loss of
generality since $Q$ has no directed cycles). For arrows $y
\xrightarrow{a} x \xrightarrow{b} \tau(x) \xrightarrow{c} z$, we
define an arrow $y \xrightarrow{(a,b,c)} z$ in $Q'_1$. We define
$\tau'$ on these arrows by $\tau'(a,b,c) = (\tau(c), \tau(b),
\tau(a))$. If this arrow is $\tau'$-invariant, we set $s(a,b,c) =
s(b)$. Also, for arrows $y \xrightarrow{a} x \xrightarrow{c} z$ with
$z \ne \tau(x)$, we define an arrow $y \xrightarrow{(a,c)} z$ in
$Q'_1$, and we do a similar thing when $x$ is replaced by
$\tau(x)$. We set $\tau'(a,c) = (\tau(c), \tau(a))$.  Any other arrows
not incident to $x$ or $\tau(x)$ are also added to $Q'_1$; $\tau'$
and $s'$ are defined as the restriction of $\tau$ and $s$ on these
arrows.

Given a symmetric representation $W$ of $Q$ of dimension $\beta$, we
define a representation $W' = \res W$ of $Q'$ of dimension $\beta'$,
where $\beta'$ is the restriction of $\beta$, as follows. First, set
$W'(y) = W(y)$ for all $y \in Q'_0$. For any arrow of the form $y
\xrightarrow{(a,b,c)} z$ we define $W'_{(a,b,c)} =
W_cW_bW_a$. Similarly, for arrows of the form $y \xrightarrow{(a,c)}
z$ we set $W'_{(a,c)} = W_cW_a$. For all other arrows, we define $W'$
to be the restriction of $W$. This gives a map as in
\eqref{compositionmap}, which is surjective by
Theorem~\ref{theorem:fund}\eqref{fund:general} and
Lemma~\ref{fund:tensors}.

Now we need to know that under \eqref{compositionmap}, determinantal
and Pfaffian semi-invariants pull back to determinantal and Pfaffian
semi-invariants, respectively.

\begin{proposition} Use the notation above. The image of a
  determinantal semi-invariant $c^{V'} \in \SSI(Q',\beta')_{\sigma'}$
  under \eqref{compositionmap} is of the form $c^V$ for some
  representation $V$ of $Q$. Similarly, the image of a Pfaffian
  semi-invariant is a Pfaffian semi-invariant.
\end{proposition}

\begin{proof} First we deal with determinantal semi-invariants. In the
  case $x = \tau(x)$, the proof of Step 2 of Theorem 1 in
  \cite{saturation} works in our case. So we get $V$ such that the
  pullback of $c^{V'}$ is $c^V$. 

  Now we deal with the case $x \ne \tau(x)$. We define an intermediate
  quiver $\tilde{Q}$ by forgetting that $Q$ is a symmetric quiver and
  deleting $x$ from $Q^\circ$ as in Step 2 of Theorem 1 in
  \cite{saturation}. Deleting $\tau(x)$ from $\tilde{Q}$ again as in
  Step 2 of Theorem 1 in \cite{saturation}, we get to
  $(Q')^\circ$. Hence we have maps
  \[
  \Rep(Q^\circ) \to \Rep(\tilde{Q}) \to \Rep((Q')^\circ).
  \]
  Restricting the composition to $\SRep(Q)$, the image is in
  $\SRep(Q')$, and we recover our restriction map $\res$. Now Step 2
  of Theorem 1 in \cite{saturation} implies that $\res^* c^{V'} = c^V$
  for some $V \in \Rep(Q, \beta)$.

  Finally, if we have a Pfaffian semi-invariant $\pf^{V'}$ (in either
  case), then $(\res^*\pf^{V'})^2 = c^V$ for some $V$, and hence
  $\res^*\pf^{V'} = \pf^V$.
\end{proof}

By induction, we may assume that $Q$ has only two vertices, so we can
reduce to the case of a generalized Kronecker quiver where the weight
is 1.

\subsubsection{The generalized Kronecker quiver.} 

We have to work with the quiver $\Theta_{p,r}^+$ (respectively
$\Theta_{p,r}^-$) which is defined to be $p$ arrows $x_-
\xrightarrow{a_i} x_+$ with $s(a_i) = 1$ (respectively $s(a_i) = -1$)
and $r$ pairs of arrows $\{b_i, b'_i\}$ such that $\tau(b_i) = b'_i$
and $\tau(b'_i) = b_i$. The dimension vector is $\beta = (n,n)$, and
the weight is given by $\sigma(x_-) = 1$.

This case can be handled similarly to the way it was handled in
\cite{saturation}. In the case of $\Theta^-_{p,r}$, with dimension
vector $(n,n)$, let $V = K^n$. The representation variety is $\hom(V,
V^*)^{\oplus r} \oplus (\bigwedge^2 V^*)^{\oplus p}$, which splits up
in a $\GL(V)$-equivariant way as $(\Sym^2 V^*)^{\oplus r} \oplus
(\bigwedge^2 V^*)^{\oplus (p+r)}$. We are only interested in the
semi-invariants of weight 1, and by \eqref{eqn:sym2plethysm}, only the
polynomial functions on $\bigwedge^2 V^*$ can contribute. We think of
this as $(\bigwedge^2 V^*) \otimes U$ with an action of $\GL(U)$ where
$U \cong K^{p+r}$. (We really have $p+2r$ arrows, so $U$ is
identifying the arrows that come in pairs.) By \eqref{eqn:cauchy}, we
have
\[
K[\bigwedge^2 V^* \otimes U] = \Sym(\bigwedge^2 V \otimes U^*) \sim
\bigoplus_\lambda \bS_\lambda(\bigwedge^2 V) \otimes \bS_\lambda(U^*),
\]
and $\bS_\lambda(\bigwedge^2 V)$ contains an $\SL(V)$-invariant of
weight 1 if and only if $n$ is even and $\lambda = (n/2)$ by
Corollary~\ref{cor:wedge2invariants}. Hence we have
\[
\SSI(\Theta_{p,r}^-, (n,n))_1 = \bigwedge^n V \otimes {\rm
  D}^{n/2}(U^*),
\]
where ${\rm D}$ denotes the divided power functor. To see this, we
just note that $\Sym^{n/2}(\bigwedge^2 V\otimes U^*) \subset
\Sym^{n/2}(V \otimes V \otimes U^*)$, and the latter module contains
$\bigwedge^{n/2} V \otimes \bigwedge^{n/2} V \otimes {\rm D}^{n/2}
U^*$ by the dual Cauchy filtration for $\bigwedge^{n/2}(V \otimes
U^*)$ \cite[Theorem 2.3.2(b)]{weyman}.

Since ${\rm D}^{n/2}(U^*)$ is a highest weight module for $\GL(U)$, it
is enough to show that its highest weight vector is represented by a
semi-invariant $\pf^V$ for some representation $V$ since these kinds
of semi-invariants are invariant under the action of $\GL(U)$.  Fix an
ordering of the arrows $a_1, \dots, a_{p+2r}$. Now define a
representation $W$ of dimension $(1,p+2r-1)$ by $W_{a_1} = 0$ and
$W_{a_i}({\bf 1}) = e_{i-1}$ for all $i>1$, where ${\bf 1}$ is a
nonzero vector of $W_-$ and $e_1, \dots, e_{p+2r-1}$ is a basis for
$W_+$. Then $c^W = \det V_{a_1}$ (up to scalar multiple) and hence
$\pf^W = \pf V_{a_1}$. Then $\pf^{W_{a_1}}$ is the desired highest
weight vector.

The situation of $\Theta^+_{p,r}$ is similar to that of
$\Theta^-_{p,r}$. The representation variety instead decomposes as
$(\Sym^2 V^*)^{\oplus (p+r)} \oplus (\bigwedge^2 V^*)^{\oplus r}$, but
we can proceed as above. 

This concludes the proof of Theorem~\ref{pfaffians}.



\section{Semi-invariants of quivers with relations.} \label{section:relations}

Suppose that the characteristic is different from 2 for this section,
except in the statements of
Proposition~\ref{proposition:pfaffianrelations} and
Theorem~\ref{theorem:SIrelations}, where the characteristic is assumed
to be 0.

\subsection{Quivers with relations.}

Let $Q$ be a symmetric quiver (without oriented cycles, as usual), and
let $KQ$ be its path algebra. We let $\fm \subset KQ$ be the two-sided
ideal generated by all arrows $a \in Q_1$. Given any two vertices $x,
y \in Q_0$ we say that a linear combination of paths from $x$ to $y$
is {\bf homogeneous}. A two-sided ideal $I \subset KQ$ is {\bf
  admissible} if $I \subseteq \fm^2$. We will also assume that $I$ is
$\tau$-invariant. We denote by $Q/I$ the quiver with relations $I$,
i.e., (symmetric) representations of $Q/I$ are (symmetric)
representations $V$ of $Q$ for which $IV = 0$. The requirement on
admissible ideals is for convenience, since having an element of
length 0 or 1 in the relations is equivalent to considering the
quotient $Q'/I'$ where $Q'$ is the quiver with the corresponding
vertices or edges removed, and $I'$ is the ideal $I$ with the
corresponding relations removed.

\subsubsection{Representations.}

The symmetric representation variety of $Q/I$ of dimension $\beta$ is
denoted $\SRep(Q/I, \beta)$. This is a closed subvariety of $\SRep(Q,
\beta)$ which is $\SG(Q, \beta)$-invariant. Hence the surjection
$K[\SRep(Q, \beta)] \to K[\SRep(Q/I, \beta)]$ is $\SG(Q,
\beta)$-equivariant. If $K$ has characteristic 0, then $\SG(Q, \beta)$
is linearly reductive, so by semisimplicity, we get a surjective map
of semi-invariants $\SSI(Q, \beta) \to \SSI(Q/I, \beta)$.

However, we are not interested in the whole variety $\SRep(Q/I,
\beta)$. In general, this variety is reducible, so let $\SRep(Q/I,
\beta)^{(1)}, \dots, \SRep(Q/I, \beta)^{(N)}$ denote its irreducible
components. 

For a representation $V$ of $Q/I$, let $\pdim V$ denote the projective
dimension of $V$ over $KQ/I$. The global dimension of $Q/I$
(abbreviated $\gldim Q/I$) is defined as the largest possible
projective dimension of a module of $Q/I$, and is at most $\#Q_0 - 1$
(this can be proven by induction on the number of vertices after
choosing an ordering of the vertices such that arrows only go from
smaller vertices to bigger ones, but we won't use this fact). Hence we
can define a modified Euler form for two representations $V$ and $W$
of $Q/I$ with dimension vectors $\alpha = \dim V$ and $\beta = \dim W$
via
\begin{align} \label{eqn:eulerform:Q/I}
\langle \alpha, \beta \rangle_I = \sum_{i \ge 0} (-1)^i \dim_K
\ext^i_{Q/I} (V, W).
\end{align}
In fact, this definition depends only on the vectors $\alpha$ and
$\beta$ \cite[Proposition III.3.13]{assem}. The indecomposable
projective modules of $Q/I$ are indexed by $Q_0$, and are given by
$P_x = P'_x / IP'_x$ where $P'_x$ is the indecomposable projective
module for $Q$ indexed by $x \in Q_0$ \cite[Lemma
III.2.4]{assem}. From this, we see that projective covers are
well-defined for representations of $Q/I$, see also \cite[Theorem
I.5.8]{assem}.

\subsubsection{Semi-invariants.} \label{section:semiinvariantsQ/I}

We denote the ring of semi-invariants $K[\SRep(Q/I,
\beta)^{(j)}]^{\SG(Q, \beta)}$ by $\SSI(Q/I, \beta)^{(j)}$. These
components need not be invariant under ${\bf G}(Q,\beta)$ if
$Q_0^{\tau_+}$ is nonempty. However, they are invariant under $\SG(Q,
\beta)$ since it is a connected group. We can define analogues of
determinantal and Pfaffian semi-invariants for symmetric quivers with
relations. When discussing semi-invariants, we will work modulo the
nilpotent radical. This remark will not affect our main application in
Section~\ref{section:flagquiver}, but we include it here for
simplicity.

Given a representation $V$ of $Q/I$, let
\[
P_1 \to P_0 \to V \to 0
\]
be a minimal projective presentation of $V$. Given any other
representation $W$, we define the map
\[
\ol{d}^V_W \colon \hom_{KQ/I}(P_0, W) \to \hom_{KQ/I}(P_1, W).
\]
If it is a square matrix, we set $\ol{c}^V_W$ to be its
determinant. In this case, we get a determinantal semi-invariant
$\ol{c}^V \in \SSI(Q/I, \beta)^{(j)}$. If this polynomial function is
a square, we likewise define the Pfaffian semi-invariant $\ol{\pf}^V
\in \SSI(Q/I, \beta)^{(j)}$. Similar to before, this is well-defined
up to a nonzero scalar since $K[\SRep(Q/I, \beta)^{(j)}]$ is a domain.

Now let $V'$ be a representation of $Q$ and let $V = V'/IV'$ be the
corresponding representation of $Q/I$. The Pfaffian semi-invariant
$\pf^{V'}$ restricts to a semi-invariant in $\SSI(Q/I, \beta)^{(j)}$
for all $j$. In particular, its image is $\ol{\pf}^{V}$, which can be
shown using the proof of \cite[Proposition 1]{relations}. So we have
the following result.

\begin{proposition} \label{proposition:pfaffianrelations} Suppose that
  the characteristic of $K$ is $0$. Let $Q$ be a symmetric quiver
  without oriented cycles and $I \subset KQ$ an admissible ideal. The
  Pfaffian semi-invariants $\ol{\pf}^V$ linearly span $\SSI(Q/I,
  \beta)^{(j)}$.
\end{proposition}

For our applications, we shall only be interested in certain kinds of
irreducible components. A component $\Rep(Q/I, \beta)^{(j)}$ is {\bf
  faithful} if whenever $x \in KQ$ annihilates all representations in
$\Rep(Q/I, \beta)^{(j)}$, we have $x \in I$. We make similar
definitions for symmetric representation varieties.

\begin{theorem}[Derksen--Weyman] \label{theorem:SIrelations} Suppose
  that the characteristic of $K$ is $0$. Let $Q$ be an acyclic quiver
  and let $I \subset KQ$ be an admissible ideal. Suppose that
  $\Rep(Q/I, \beta)^{(j)}$ is a faithful component of $\Rep(Q/I,
  \beta)$. If $\ol{c}^V$ is nonzero on $\Rep(Q/I, \beta)$, then $\pdim
  V \le 1$.
\end{theorem}

\begin{proof} See \cite[Theorem 1]{relations}. \end{proof}

\subsubsection{Reminders on semicontinuity and
  genericity.} \label{section:semicontinuity} 

Let $X$ be a topological space and $P$ be a partially ordered set. A
function $f \colon X \to P$ is {\bf upper semicontinuous} if the sets
$X_{< n} = \{x \in X \mid f(x) < n\}$ are open for all $n \in P$. We
shall mostly be interested in the case when $P = \N$, so that the set
where $f$ attains its minimum is open, or the case when $P = \N^{Q_0}$
for some vertex set $Q_0$ with the partial order $(d_x)_{x \in Q_0}
\le (d'_x)_{x \in Q_0}$ if and only if $d_x \le d'_x$ for all $x \in
Q_0$. We list here some functions on representation varieties that are
upper semicontinuous with references for proofs. We will use these
facts without explicit mention.

Let $Q/I$ be a quiver with relations and let $X = \Rep(Q/I, \alpha)$
and $Y = \Rep(Q/I, \beta)$ for some dimension vectors $\alpha$ and
$\beta$. The function $\pdim \colon X \to \N$ that assigns a module to
its projective dimension is upper semicontinuous \cite[Lemma
2.1]{salmeron}. By duality, the same is true for injective
dimension. For any given $i$, the function $\ext^i \colon X \times Y
\to \N$ given by $(M,N) \mapsto \dim_K \ext^i_{Q/I}(M,N)$ is upper
semicontinuous \cite[Lemma 1.2]{salmeron}. In particular, if we fix a
representation $M \in X$, the function $N \mapsto \dim_K
\ext^i_{Q/I}(M,N)$ is upper semicontinuous, and there is a similar
statement when fixing the other argument. Given two irreducible
components $C \subseteq X$ and $C' \subseteq Y$, we will use
$\ext^i(C, C')$ to denote the minimum value of $\ext^i$ restricted to
$C \times C'$. In this case, the minimum is attained on an open dense
subset.

For representations $M$ and $N$, the function $\hom_{Q/I}(M,N) \to
\N^{Q_0}$ given by $\phi \mapsto \dim \ker \phi$ is upper
semicontinuous. Hence there is a unique maximal dimension vector
$\gamma$ such that the set of linear maps $M \to N$ with rank $\gamma$
is open. We call this the {\bf generic rank} for $M$ and $N$. See
\cite[Lemma 5.1]{generalreps} and its proof for details. Similarly, we
can form the subbundle
\[
\hom(Q/I, C, C') = \{(\phi, M, N) \in \hom_K(K^\alpha, K^\beta) \times
X \times Y \mid \phi \in \hom_{Q/I}(M,N) \},
\]
and the function that assigns to a triple $(\phi, M, N)$ the dimension
of $\ker \phi$ is upper semicontinuous, so we can define the generic
rank for two components $C$ and $C'$. We can also mix and match
modules and irreducible components.

Finally, given a dominant morphism of two irreducible varieties $f
\colon X \to Y$, the set $\{y \in Y \mid \dim f^{-1}(y) = \dim X -
\dim Y\}$ contains a dense open subset of $Y$.

\subsection{Global dimension 2.}

Suppose that the characteristic is arbitrary and also that $\gldim Q/I
\le 2$ for this section. Let $Q_2$ be a set of homogeneous minimal
relations, i.e., for all $r \in Q_2$, we have that $r$ is not
contained in the ideal generated by $Q_2 \setminus \{r\}$. For $r \in
Q_2$, let $tr$ and $hr$ be the beginning and ending, respectively, of
the paths which are the summands in $r$. There is a canonical
resolution just as in the case of a quiver $Q$:
\begin{align} \label{eqn:relationscanonical}
0 \to \bigoplus_{r \in Q_2} V(tr) \otimes P_{hr} \to \bigoplus_{a \in
  Q_1} V(ta) \otimes P_{ha} \to \bigoplus_{x \in Q_0} V(x) \otimes P_x
\to V \to 0.
\end{align}
It follows that the Euler form for $Q/I$ \eqref{eqn:eulerform:Q/I} can
be defined as
\[
\langle \alpha, \beta \rangle_I = \sum_{x \in Q_0} \alpha(x) \beta(x)
- \sum_{a \in Q_1} \alpha(ta) \beta(ha) + \sum_{r \in Q_2} \alpha(tr)
\beta(hr).
\]

\begin{remark} \label{remark:weights} Suppose $\dim V = \alpha$ and
  $\pdim V \le 1$. We claim that the weight $\sigma_\alpha$ of the
  semi-invariant $\ol{c}^V$ on $\Rep(Q/I, \beta)$ is given by
  $\sigma_\alpha(x) = \langle \alpha, \eps_x \rangle_I$.

  First, let $0 \to P_1 \to P_0 \to V \to 0$ be a projective
  resolution, and choose $W \in \Rep(Q/I, \beta)$. For $x \in Q_0$, we
  know by \eqref{eqn:relationscanonical} that the number of times that
  $P_x$ appears in $P_0$ minus the number of times that $P_x$ appears in
  $P_1$ is $\langle \alpha, \eps_x \rangle_I$. Hence when we apply
  $\hom(-, W)$ to the projective resolution, we can use the proof of
  \cite[Lemma 1.4]{schofield} to conclude our desired result. 
\end{remark}

\begin{remark} We will use the following facts repeatedly without
  explicit mention. If $V$ is a module with $\pdim V \le 1$, then
  $\pdim V' \le 1$ for any submodule $V' \subseteq V$. To see this,
  first note that $\pdim V \le 1$ is equivalent to the functor
  $\ext^2(V, -)$ being identically 0. Using that $\gldim Q/I \le 2$,
  we see that $\ext^2(V, W) = 0$ implies that $\ext^2(V', W) = 0$ for
  all modules $W$. Dually, if $\idim V \le 1$, then $\idim V/V' \le 1$
  for any submodule $V' \subseteq V$.
\end{remark}

\begin{proposition} \label{prop:dimension} Suppose $\gldim Q/I \le 2$,
  and let $M$ be a representation of $Q/I$ such that $\ext^2(M,M) =
  0$. Set $\alpha = \dim M$. Then $M$ is a nonsingular point of
  $\Rep(Q/I, \alpha)$, and 
  \begin{align*}
    \dim_M \Rep(Q/I, \alpha) &= \sum_{a \in Q_1} \alpha(ta) \alpha(ha)
    - \sum_{r \in Q_2} \alpha(tr) \alpha(hr).
  \end{align*}
  Here $\dim_x X$ means the local dimension of $X$ at $x \in X$.
\end{proposition}

\begin{proof} See \cite[Proposition 2.2]{bobinski}.
\end{proof}

\begin{definition}
Set 
\begin{align*}
  Q^e_0 &= \{0,1\} \times Q_0,\\ 
  Q^e_1 &= \{((0,ta), (0,ha)) \mid a \in
  Q_1\} \cup \{((1,ta), (0,ha)) \mid a \in Q_1\} \cup \{((1,ta),
  (1,ha)) \mid a \in Q_1\}.
\end{align*}
There is an abuse of notation here since two arrows might have the
same head and tail, but we hope the meaning is clear. Define an ideal
of relations $I^e$ as follows. For every relation $\sum_p \lambda_p
a_{p_{d(p)}} \cdots a_{p_2} a_{p_1}$ in $I$, we take the homogeneous
components of the relations $\sum_p \lambda_p a'_{p_{d(p)}} \cdots
a'_{p_2} a'_{p_1}$ where $a'_x = ((0,ta), (0,ha)) + ((1,ta), (0,ha)) +
((1,ta), (1,ha))$. We call $Q^e/I^e$ the {\bf extension quiver} of
$Q/I$.
\end{definition}

\begin{remark}
  Given a representation $V$ of $Q^e/I^e$, there is an associated
  representation $V'$ of $Q/I$ along with a choice of submodule $V''
  \subset V'$ by setting $V'_x = V_{(0,x)} \oplus V_{(1,x)}$ and
  $V''_x = V_{(0,x)}$. Conversely, given an inclusion of $Q/I$-modules
  $V'' \subset V'$, one can associate to it a representation $V$ of
  $Q^e/I^e$ by picking a basis for $V''$ and extending it to a basis
  for $V'$. There is some ambiguity about this choice of basis, but it
  will not affect our discussions.
\end{remark}

\begin{proposition} \label{proposition:extensions} If $\gldim Q/I \le
  2$, then $\gldim Q^e/I^e \le 2$.
\end{proposition}

\begin{proof} It will be enough to show that every simple
  representation $S_{(n,x)}$ has projective dimension at most 2. The
  projective modules $P_{(0,x)}$ are supported in $\{0\} \times Q$, so
  $\pdim_{Q^e/I^e} S_{(0,x)} = \pdim_{Q/I} S_x \le 2$ by
  assumption. Otherwise, we claim that
  \[
  0 \to \bigoplus_{r \in Q_2,\ tr = x} P_{(0,hr)} \oplus P_{(1,hr)}
  \xrightarrow{d_2} \bigoplus_{a \in Q_1,\ ta = x} P_{(0,ha)} \oplus
  P_{(1,ha)} \xrightarrow{d_1} P_{(1,x)} \xrightarrow{d_0} S_{(1,x)}
  \to 0.
  \]
  is a projective resolution. It is clear that $d_0$ is surjective and
  that $\im d_1 = \ker d_0$. An element in $\ker d_1$ is a linear
  combination of paths $p_{(n,a)}$ starting at $(n,ha)$ for various $n
  \in \{0,1\}$ and $a \in Q_1$ with $ta = x$, all of which end at a
  common vertex $(m,y)$, such that appending $a$ to the beginning of
  $p_{(n,a)}$ and taking the sum gives 0, i.e., is a relation between
  $(1,x)$ and $(m,y)$ in $I^e$. Hence $\im d_2 = \ker d_1$. 

  Finally, we show that $d_2$ is injective. An element in $\ker d_2$
  is a linear combination of paths $p_{(n,r)}$ starting at $(n,hr)$
  for various $n \in \{0,1\}$ and $r \in Q_2$ with $tr = x$ and $r =
  \sum_p \lambda_p a_{p_{d(p)}} \cdots a_{p_2} a_{p_1}$, all of which
  end at a common vertex $(m,y)$, such that appending $\sum_p
  \lambda_p a_{p_{d(p)}} \cdots a_{p_2}$ to the beginning of each
  $p_{(n,r)}$ and taking the sum gives 0. If this element is nonzero,
  this will imply that there is a minimal relation in $I^e$ such that
  the beginning of one of its paths coincides with the ending of one
  of the paths of $r$. This contradicts the fact that $\gldim Q/I \le
  2$ and our definition of $I^e$.
\end{proof}

\begin{proposition} \label{proposition:extensions2} Let $N \subset M$
  be representations of $Q/I$ so that $M' = (N \subset M)$ is
  naturally a representation of $Q^e/I^e$. If $\ext^2(M/N, M/N) =
  \ext^2(M, N) = 0$, then $\ext^2_{Q^e/I^e}(M', M') = 0$.
\end{proposition}

\begin{proof} Let $P_\bullet \to M' \to 0$ be the canonical resolution
  \eqref{eqn:relationscanonical} of $M'$. Write $P_i^{(n)}$ for the
  summands of $P_i$ starting at vertices of $\{n\} \times Q_0$. Since
  $P_3 = 0$, it is enough to show that $\hom(P_1, M') \to \hom(P_2,
  M')$ is surjective. This homomorphism can be broken up into two
  pieces.  The first piece is $\hom(P_1^{(0)}, M') \to \hom(P_2^{(0)},
  M')$, which is the same thing as $\hom(P_1^{(0)}, N) \to
  \hom(P_2^{(0)}, N)$, whose cokernel can be identified with
  $\ext^2_{Q/I}(M, N) = 0$. Hence this part is surjective, and the
  second piece is $\hom(P_1^{(1)}, M') \to \hom(P_2^{(1)}, M')$, which
  is the same as $\hom(P_1^{(1)}, M/N) \to \hom(P_2^{(1)}, M/N)$. But
  this map is also surjective since the cokernel of this map can be
  identified with $\ext^2_{Q/I}(M/N, M/N) = 0$.
\end{proof}

\begin{example} 
  Let $Q/I$ be the quiver $1 \xrightarrow{\alpha} 2
  \xrightarrow{\beta} 3$ with the relation $\beta \alpha = 0$. Let $M$
  be the representation $K \to 0 \to K$ and let $N$ be the
  subrepresentation $0 \to 0 \to K$. The projective resolution for $M'
  = (N \subset M)$ for $Q^e/I^e$ is
  \[
  0 \to P_{(1,3)} \oplus P_{(0,3)} \to P_{(0,2)} \oplus P_{(1,2)} \to
  P_{(1,1)} \oplus P_{(0,3)} \to M' \to 0.
  \]
  The first map in the above proof becomes
  \begin{align*}
  0 = \hom(P_1^{(0)}, N) &\to \hom(P_2^{(0)}, N) = K
  \end{align*}
  which we naturally think of as $\ext^2_{Q/I}(M,N) = K$ so that
  $\ext^2(M',M') = K$. In this example, it is easy to produce a
  nonzero element of $\ext^2(M',M')$:
  \[
  0 \to \begin{array}{ccc} K & 0 & 0 \\ 0 & 0 & K \end{array}
  \to \begin{array}{ccc} K & 0 & 0 \\ 0 & 0 & 0 \end{array}
  \oplus \begin{array}{ccc} 0 & K & 0 \\ 0 & K & K \end{array}
  \to \begin{array}{ccc} K & K & 0 \\ 0 & K & 0 \end{array}
  \oplus \begin{array}{ccc} 0 & 0 & 0 \\ 0 & 0 & K \end{array}
  \to \begin{array}{ccc} K & 0 & 0 \\ 0 & 0 & K \end{array} \to 0,
  \]
  where we have drawn $Q^e/I^e$ as
  \[
  \xymatrix @-1.2pc { (1,1) \ar[r] \ar[rd] & (1,2) \ar[r] \ar[rd] & (1,3) \\
    (0,1) \ar[r] & (0,2) \ar[r] & (0,3) },
  \]
  and the arrows in each direct summand are nonzero whenever possible.
  If we instead took $N = (K \to 0 \to 0)$, then $\ext^2(M', M') = 0$.
\end{example}

\subsection{Quiver Grassmannians.} \label{section:quivergrass}

Suppose that the characteristic is arbitrary except in
Proposition~\ref{proposition:symmetricgrassmannian} where the
characteristic is different from 2. We continue to use the notation of
the previous section. Given two nonnegative integers $k \le n$,
$\Gr(k,n)$ denotes the {\bf Grassmannian}, a variety parametrizing the
$k$-dimensional subspaces of a fixed $n$-dimensional vector space. It
is a projective variety of dimension $k(n-k)$. Given a quiver with
relations $Q/I$, a module $W$, and a dimension vector $\gamma$, we let
$\Gr(\gamma, W)$ be the {\bf quiver Grassmannian}, which is the
variety of submodules of $W$ of dimension $\gamma$. This is a
projective variety and the calculation of its dimension in some
special cases will be the subject of our attention for this
section. We will only be using topological properties of this variety,
so its scheme structure will not play a role in this paper.

Let $\beta$ be a dimension vector and let $\Rep(Q/I, \beta)'$ be the
closure inside $\Rep(Q/I, \beta)$ of points corresponding to
representations with injective dimension at most 1, and let $W$ be
such a representation. Fix a dimension vector $\gamma \le \beta$, and
let $\Gr(\gamma, W)'$ denote the subvariety of the quiver Grassmannian
consisting of submodules $U \subseteq W$ such that $\dim U = \gamma$
and $\idim U \le 1$.

\begin{theorem} \label{theorem:grassmannian} With the notation above,
  suppose that $\Gr(\gamma, W)'$ is nonempty for a general $W$ with
  $\idim W \le 1$. Then there is a nonempty open set $\Omega \subseteq
  \Rep(Q/I, \beta)'$ such that $\dim \Gr(\gamma, W)' = \langle \gamma,
  \beta - \gamma \rangle_I$ for $W \in \Omega$.
\end{theorem}

\begin{proof}
  Let $\Gr(\gamma, \beta) = \prod_{x \in Q_0} \Gr(\gamma(x),
  \beta(x))$ denote a product of ordinary Grassmannians and let
  $\Rep(Q/I, \gamma \subset \beta)$ be the closure inside of
  $\Rep(Q/I, \beta)' \times \Gr(\gamma, \beta)$ of the set of points
  $(W, V)$ such that $\idim W \le 1$ and such that the subspaces
  determined by $V$ form a submodule of $W$ with $\idim V \le
  1$. Consider the projection $\pi_2 \colon \Rep(Q/I, \gamma \subset
  \beta) \to \Gr(\gamma, \beta)$. The fiber over a point $V$ consists
  of all representation structures on $K^\beta$ such that $V$ forms a
  submodule with $\idim V \le 1$. 

  To better describe this fiber, first choose a splitting $K^\beta =
  K^\gamma \oplus K^{\beta - \gamma}$ of vector spaces. Define a
  dimension vector $\alpha$ of $Q^e/I^e$ by $\alpha(0,x) = \gamma(x)$
  and $\alpha(1,x) = (\beta - \gamma)(x)$.

  Let $Z \subset \Rep(Q^e/I^e, \alpha)$ be the subvariety consisting
  of modules $M$ such that $\idim M \le 1$ when thought of as a module
  over $Q/I$. There is a map $p \colon Z \to \Rep(Q/I, \gamma)$ which
  sends $W$ to the restriction of $W$ to the subquiver $\{0\} \times
  Q_0$. The subvariety $\Rep(Q/I, \gamma)'$ of $\Rep(Q/I, \gamma)$
  consisting of representations with injective dimension at most 1 is
  open (and nonempty by our assumptions), so $Z' = p^{-1}(\Rep(Q/I,
  \gamma)')$ can be identified with $\pi_2^{-1}(V)$. We know that
  $\gldim Q^e/I^e \le 2$ by
  Proposition~\ref{proposition:extensions}. So by
  Proposition~\ref{proposition:extensions2}, the local dimension at
  every point of $Z'$ is given by
  Proposition~\ref{prop:dimension}. Hence we have
  \begin{align*}
    \dim \pi_2^{-1}(V) &= \sum_{x \in Q^e_0} \alpha(x)^2 - \langle
    \alpha, \alpha \rangle_{I^e}
    = \sum_{x \in Q_0} (\beta - \gamma)(x)\beta(x) + \sum_{x \in Q_0}
    \gamma(x)^2 - \langle \beta - \gamma, \beta \rangle_I - \langle
    \gamma, \gamma \rangle_I,
  \end{align*}
  and thus
  \begin{align*}
    \dim \Rep(Q/I, \gamma \subset \beta) &= \dim \Gr(\gamma, \beta) +
    \dim \pi_2^{-1}(V)\\
    &= \sum_{x \in Q_0} \gamma(x) (\beta - \gamma)(x)
    + \dim \pi_2^{-1}(V) \\
    &= \sum_{x \in Q_0} \beta(x)^2 - \langle \beta - \gamma, \beta
    \rangle_I - \langle \gamma, \gamma \rangle_I.
  \end{align*}

  By Proposition~\ref{prop:dimension}, $\Rep(Q/I, \beta)'$ is
  equidimensional of dimension $\sum_{x \in Q_0} \beta(x)^2 - \langle
  \beta, \beta \rangle_I$. For $W \in \Rep(Q/I, \beta)'$, we have
  $\pi_1^{-1}(W) = \Gr(\gamma, W)'$. Hence there is an open set
  $\Omega$ such that for $W \in \Omega$, one has
  \begin{align*}
    \dim \Gr(\gamma, W)' &= \dim \Rep(Q/I, \gamma \subset \beta) -
    \dim \Rep(Q/I, \beta)' \\
    &= -\langle \beta - \gamma, \beta \rangle_I - \langle \gamma,
    \gamma \rangle_I + \langle \beta, \beta \rangle_I = \langle
    \gamma, \beta - \gamma \rangle_I. \qedhere
  \end{align*}
\end{proof}

\begin{remark} 
  We can also consider the dual situation of a representation $W$ with
  $\pdim W \le 1$ and its quiver Grassmannian $\Gr(W, \gamma)'$ of
  $\gamma$-dimensional quotients $U$ with $\pdim U \le 1$. Following
  the above proof will also show that $\dim \Gr(W, \gamma)' = \langle
  \beta - \gamma, \gamma \rangle_I$ for generic $W$, or we can apply
  the duality functor $\hom_K(-,K)$.
\end{remark}

We need to calculate the dimension of quiver Grassmannians when $W$ is
a symmetric representation of a triple flag quiver (see
Section~\ref{section:flagquiver}). 

\begin{proposition} \label{proposition:symmetricgrassmannian} Suppose
  that the characteristic is different from $2$. Let $\QQ$ be an
  $r$-uple flag quiver and $\beta^\delta$ denote the corresponding
  dimension vector. Then for general symmetric $W$ where the
  appropriate maps are injective and surjective, and $\gamma \le
  \beta^\delta$ such that $\gamma(u) \ge \gamma(X_n) +
  \gamma(\tau(X_n))$ for all $X \in \{x,y,z\}$, we have $\dim
  \Gr(\gamma, W)' = \langle \gamma, \beta - \gamma \rangle_I$.
\end{proposition}

\begin{proof} Let $Q/I$ be the following symmetric quiver with
  relations:
  \[
  K^n \xrightarrow{a} K^{2n+\delta} \xrightarrow{\tau(a)} K^n
  \]
  where $\tau(a)a = 0$ and $K^{2n+\delta}$ has a bilinear form. Let
  $\beta$ be this dimension vector and pick $\gamma \le \beta$ such
  that $\gamma_2 \ge \gamma_1 + \gamma_3$ and $\Gr(\gamma, W)'$ is
  nonempty for general symmetric $W$. We use the setup from the proof
  of Theorem~\ref{theorem:grassmannian}:
  \[
  \SRep(Q/I, \beta) \xleftarrow{\pi_1} \SRep(Q/I, \gamma \subset
  \beta) \xrightarrow{\pi_2} \Gr(\gamma, \beta)
  \]
  Then $\pi_2$ is surjective by assumption on $\gamma$, so
  \[
  \dim \Gr(\gamma, W)' = \dim \Gr(\gamma, \beta) + \dim \pi_2^{-1} -
  \dim \SRep(Q/I, \beta),
  \]
  where $\dim \pi_2^{-1}$ is the dimension of a general fiber of
  $\pi_2$.

  We can show directly that $\dim \Gr(\gamma, W)' = \langle \gamma,
  \beta - \gamma \rangle_I$ when $W_a$ is injective and $W_{\tau(a)}$
  is surjective. To get a $\gamma$-dimensional submodule $U$, we first
  choose a $\gamma_1$-dimensional subspace in $K^n$, which gives
  $\gamma_1(n-\gamma_1)$ dimensions of choices. This determines its
  image in $K^{2n+\delta}$. Set $d = \gamma_2 - \gamma_1 -
  \gamma_3$. We need $\dim \ker U_{\tau(a)} = \gamma_2 - \gamma_3$ and
  $U_a(K^{\gamma_1}) \subset \ker U_{\tau(a)}$, so to pick the rest of
  the kernel, we pick a $d$-dimensional subspace of $\ker W_{\tau(a)}
  / U_a(K^{\gamma_1})$. Since $\dim \ker W_{\tau(a)} = n+\delta$, this
  gives us $d(n+\delta-\gamma_1-d)$ dimensions of choice.

  Finally, we choose a $\gamma_3$-dimensional subspace in
  $K^{2n+\delta} / (K^{\gamma_1} + K^d)$ disjoint from the image of
  $W_a(K^n)$. A generic choice of subspace works, so we get
  $\gamma_3(2n + \delta - \gamma_2)$ dimensions of choice. This
  determines the $\gamma_3$-dimensional subspace in the last
  $K^n$. Thus,
  \begin{align*}
    \dim \Gr(\gamma, W)' &= \gamma_1(n-\gamma_1) + d(n + \delta +
    \gamma_3 - \gamma_2) + \gamma_3(2n+\delta - \gamma_2)\\
    &= \gamma_1(n - \gamma_1) + (\gamma_2 -
    \gamma_1)(2n+\delta-\gamma_2) + d(\gamma_3-n)\\
    &= \gamma_1(n-\gamma_1) + \gamma_2(2n+\delta-\gamma_2) +
    \gamma_3(n-\gamma_3) \\
    &\quad -\gamma_1(2n+\delta-\gamma_2) - \gamma_2(n-\gamma_3) +
    \gamma_1(n-\gamma_3)\\
    &= \langle \gamma, \beta - \gamma \rangle_I.
  \end{align*}

  From this, it is a quick check to show that $\dim \Gr(\gamma, W)' =
  \langle \gamma, \beta - \gamma \rangle_I$ for the symmetric quiver
  with relations
  \[
  K^1 \to K^2 \to \cdots \to K^n \xrightarrow{a} K^{2n+\delta}
  \xrightarrow{\tau(a)} K^n \to \cdots \to K^1.
  \]

  We get the formula
  \[
  \dim \pi_2^{-1} = -\sum_{a \in Q_1} \gamma(ta) (\beta - \gamma)(ha)
  + \sum_{r \in Q_2} \gamma(tr) (\beta - \gamma)(hr) + \dim \SRep(Q/I,
  \beta). 
  \]
  This formula is ``additive'' with respect to the number $r$ of arms
  (the formula only involves arrows and relations, and not vertices),
  and the general fiber $\pi_2^{-1}$ when there are $r$ arms can be
  thought of as the direct product of $r$ copies of the general fiber
  when there is 1 arm. So we can work as in the proof of
  Theorem~\ref{theorem:grassmannian} to conclude that $\dim
  \Gr(\gamma, W)' = \langle \gamma, \beta - \gamma \rangle_I$ for $r$
  arms.
\end{proof}

\subsection{Calculating $\ext^1_{Q/I}$.}

Suppose that the characteristic is arbitrary in this section. Recall
from Section~\ref{section:quivergrass} that $\Gr(\gamma, W)'$ denotes
the variety of submodules $U$ of $W$ with $\dim U = \gamma$ and $\idim
U \le 1$.

\begin{lemma} \label{lemma:induction} Suppose $\gldim Q/I \le 2$, and
  let $\alpha$ be a dimension vector and let $W$ be a representation
  of dimension $\beta$. Assume that the following conditions hold.
  \begin{compactenum}[\rm (a)]
  \item \label{item:injective} There is an irreducible component $C
    \subseteq \Rep(Q/I, \alpha)$ such that a general representation
    $V$ satisfies $\idim V \le 1$ and $\pdim V \le 1$.  Let $\gamma$
    be the generic rank of a homomorphism between a representation of
    $C$ and $W$.
  \item \label{item:grassmanndim} $\idim W \le 1$ and $\dim
    \Gr(\gamma, W)' = \langle \gamma, \beta - \gamma \rangle_I$.
  \end{compactenum}
  Then $\dim \ext^1(C, W) = -\langle \alpha - \gamma, \beta - \gamma
  \rangle_I$.
\end{lemma}

\begin{proof}
  The proof proceeds like the proof of \cite[Theorem
  5.2]{generalreps}. Define 
  \[
  \hom(K^\alpha, \gamma, W) = \{(\phi, U) \in \hom_K(K^\alpha, W)
  \times \Gr(\gamma, W)' \mid \phi(K^\alpha) = U\}.
  \] 
  For $U \in \Gr(\gamma, W)'$, the fiber of $U$ in $\hom(K^\alpha,
  \gamma, W)$ can be identified with an open subset of
  $\hom_K(K^\alpha, U)$, so has dimension $\sum_{x \in Q_0} \alpha(x)
  \gamma(x)$. So by \eqref{item:grassmanndim}, we have
  \[
  \dim \hom(K^\alpha, \gamma, W) = \langle \gamma, \beta - \gamma
  \rangle_I + \sum_{x \in Q_0} \alpha(x) \gamma(x).
  \]

  Now define 
  \[
  \hom(Q, \alpha, \gamma, W) = \{(M, \phi, U) \in C \times
  \hom(K^\alpha, \gamma, W) \mid \phi \in \hom_{Q/I}(M, U) \},
  \]
  and let $\pi_1$ and $\pi_2$ be the projections of $\hom(Q, \alpha,
  \gamma, W)$ to $C$ and $\hom(K^\alpha, \gamma, W)$, respectively. A
  point in $\pi_2^{-1}(\phi, U)$ is a collection of linear maps $\{M_a
  \colon K^{\alpha(ta)} \to K^{\alpha(ha)}\}_{a \in Q_1}$ that lift
  the maps in $U$.

  We can rephrase these lifts as certain representations of $Q^e/I^e$
  of dimension $\alpha^e$ where $\alpha^e(0,x) = (\alpha - \gamma)(x)$
  and $\alpha^e(1,x) = \gamma(x)$ for $x \in Q_0$. A general
  representation $M \in \Rep(Q^e/I^e, \alpha^e)$ satisfies
  $\ext^2(M,M) = 0$ by \eqref{item:injective} and
  Proposition~\ref{proposition:extensions2}. So by
  Proposition~\ref{prop:dimension},
  \[
  \dim \Rep(Q^e/I^e, \alpha^e) = \sum_{x \in Q_0} \alpha(x)(\alpha -
  \gamma)(x) - \langle \alpha, \alpha - \gamma \rangle_I + \sum_{x \in
    Q_0} \gamma(x)^2 - \langle \gamma, \gamma \rangle_I
  \]
  There is a map $q \colon \Rep(Q^e/I^e, \alpha^e) \to \Rep(Q/I,
  \gamma)$ which sends a representation to its quotient which is
  supported on $\{1\} \times Q_0$. Let $N \in \Rep(Q/I, \gamma)$ be a
  representation which appears as the image of a map $V \to W$ for
  general $V \in C$. Let $C' \subseteq \Rep(Q/I, \gamma)$ be the
  irreducible component containing $N$. Then the general rank of a
  homomorphism between $C'$ and $\Rep(Q/I, \beta)$ is $\gamma$, so we
  may rechoose the above $N$ so that $\dim q^{-1}(N) = \dim
  \Rep(Q^e/I^e, \alpha^e) - \dim C'$. This also shows that the image
  of $\pi_2$ contains a nonempty open set. Also, $\idim N \le 1$ since
  it is a quotient of $V$, and we can identify $q^{-1}(N)$ with
  $\pi_2^{-1}(\phi, U)$, so
  \begin{align} \label{eqn:pi2}
    \begin{split}
      \dim \hom(Q, \alpha, \gamma, W) &= \dim \hom(K^\alpha, \gamma,
      W) + \dim \pi_2^{-1}(\phi, U) \\
      &= \dim \hom(K^\alpha, \gamma, W) + \dim \Rep(Q^e/I^e, \alpha^e)
      - \dim C' \\
      &= \langle \gamma, \beta - \gamma \rangle_I + \sum_{x \in Q_0}
      \alpha(x) \gamma(x)\\
      & \quad + \sum_{a \in Q_1} \alpha(ta) (\alpha - \gamma)(ha) -
      \sum_{r \in Q_2} \alpha(tr) (\alpha - \gamma)(hr).
    \end{split}
  \end{align}

  On the other hand, there is an open subset $\Omega \subseteq C$ that
  is contained in the image of $\pi_1$ by definition of $\gamma$. So
  $\dim \pi_1^{-1}(\Omega) = \dim \hom(Q, \alpha, \gamma,
  W)$. Furthermore, for general $M \in \Omega$ such that $\dim
  \hom_{Q/I}(M,W) = \hom(C,W)$, we have that the fiber $\pi_1^{-1}(M)$
  is a dense open subset of $\hom_{Q/I}(M, W)$. So we get
  \begin{align} \label{eqn:pi1}
    \begin{split}
      \dim \hom(Q, \alpha, \gamma, W) &= \dim \Rep(Q/I, \alpha)
      + \dim \pi_1^{-1}(M)\\
      &= \sum_{a \in Q_1} \alpha(ta) \alpha(ha) - \sum_{r \in Q_2}
      \alpha(tr) \alpha(hr) + \hom(C, W).
    \end{split}
  \end{align}
  Putting together \eqref{eqn:pi2} and \eqref{eqn:pi1} we get
  \[
  \langle \gamma, \beta - \gamma \rangle_I + \langle \alpha, \gamma
  \rangle_I = \hom(C, W)
  \]
  Now $\ext^2(C, W) = 0$ by \eqref{item:grassmanndim}, so we can
  rewrite this equality as
  \[
  \langle \alpha, \beta \rangle_I + \ext^1(C, W) = \langle
  \gamma, \beta - \gamma \rangle_I + \langle \alpha, \gamma \rangle_I,
  \]
  from which we conclude that $\ext^1(C, W) = -\langle \alpha -
  \gamma, \beta - \gamma \rangle_I$. 
\end{proof}

\begin{theorem} \label{theorem:vanishingext} Let $Q/I$ be a quiver
  with relations and let $\alpha$ and $\beta$ be dimension
  vectors. Assume that the conditions of Lemma~\ref{lemma:induction}
  hold and use the same notation. Then for generic $V \in C$, we have
  $\dim \ext^1_{Q/I}(V, W) = \max_{\beta'} -\langle \alpha, \beta'
  \rangle_I$ where the maximum is over all dimension vectors $\beta'$
  of factor modules of $W$.
\end{theorem}

\begin{proof} 
  Pick $V \in C$ so that a generic map $V \to W$ has rank $\gamma$. We
  always have the inequality $\dim \ext^1(V, W) \ge \max_{\beta'}
  -\langle \alpha, \beta' \rangle_I$, so the content of the result is
  that there exists a factor module $W'$ of $W$ such that $\dim
  \ext^1(V, W) = -\langle \alpha, \dim W' \rangle_I$. 


  If $\gamma = 0$, then $\hom(V, W) = 0$ and hence $\ext^1(V, W) =
  -\langle \alpha, \beta \rangle_I$ since $\ext^2(V, W) = 0$ by
  Lemma~\ref{lemma:induction}\eqref{item:grassmanndim}. Otherwise if
  $\gamma \ne 0$, pick a map of rank $\gamma$, and let $V'$ be the
  kernel and $W'$ be the cokernel. Let $C_{V'}$ and $C_{W'}$ be the
  irreducible components containing $V'$ and $W'$, respectively. Now
  $\idim W' \le 1$ since $\idim W \le 1$, so $\ext^2(V', W') = 0$,
  which means that $\dim \ext^1(V', W') \ge -\langle \alpha - \gamma,
  \beta - \gamma \rangle_I$. We also have surjections
  \[
  \ext^1(V, W) \surj \ext^1(V, W') \surj \ext^1(V', W'),
  \]
  the first because $\pdim V \le 1$, and the second because $\idim W'
  \le 1$. By Lemma~\ref{lemma:induction}, $\dim \ext^1(V, W) =
  -\langle \alpha - \gamma, \beta - \gamma \rangle_I$, which forces
  $\dim \ext^1(V', W') = -\langle \alpha - \gamma, \beta - \gamma
  \rangle_I$, which is its minimal possible value, so $\ext^1(C_{V'},
  C_{W'}) = -\langle \alpha - \gamma, \beta - \gamma \rangle_I =
  \ext^1(C, W)$. This also implies that the above surjections are
  isomorphisms so that $\ext^1(C, C_{W'}) \le \ext^1(C, W)$. But we
  know that $\ext^1(C_{V'}, C_{W'}) \le \ext^1(C, C_{W'})$ since
  $\idim W' \le 1$. In particular, $\ext^1(C, C_{W'}) = \ext^1(C, W)$.

  Now a general representation $C_{W'}$ satisfies the assumptions of
  Lemma~\ref{lemma:induction}: the fact that $\idim C_{W'} \le 1$ we
  have already mentioned; we can apply
  Theorem~\ref{theorem:grassmannian} because if $\gamma'$ is the
  generic rank of a map $\phi \colon V'' \to W'$, where $V'' \in C$ is
  generic, then $\Gr(\gamma', W')'$ is nonempty because $\phi(V'')$ is
  a quotient of $V''$ and hence $\idim \phi(V'') \le 1$. Therefore, by
  induction on $\dim W$, we conclude that for general $X \in C_{W'}$,
  $\dim \ext^1(V, X) = -\langle \alpha, \beta'' \rangle_I$ where
  $\beta'' = \dim W''$ for some factor module $W''$ of $X$. Since a
  general representation in this component has a factor module of
  dimension $\beta''$, the same is true for $W'$ \cite[Lemma
  3.1]{generalreps}, and hence is also true for $W$.
\end{proof}

\section{Orthogonal and symplectic Littlewood--Richardson
  coefficients.} \label{section:bcdLRcoeff}

\subsection{Representation theory of the classical
  groups.} \label{sec:classicalrepn} 

Suppose that the characteristic is 0. Let $E$ be a
$(2n+\delta)$-dimensional vector space over $K$ (where $\delta \in
\{0,1\}$) and let $\omega$ be a nondegenerate symplectic or symmetric
bilinear form on $E$. Let $G$ be the subgroup of $\SL(E)$ which
preserves $\omega$. In order to be precise let us just list the cases:
\begin{compactenum}
\item Case $\mathrm{B}_n$: We have $\delta = 1$, $\omega$ is symmetric,
  $G = \SO(E) \cong \SO(2n+1)$.
\item Case $\mathrm{C}_n$: We have $\delta = 0$, $\omega$ is
  skew-symmetric, $G = \Sp(E) \cong \Sp(2n)$.
\item Case $\mathrm{D}_n$: We have $\delta = 0$, $\omega$ is symmetric,
  $G = \SO(E) \cong \SO(2n)$.
\end{compactenum}

We identify the weight lattice of $G$ with $\Z^n = \Z\langle \eps_1,
\dots, \eps_n \rangle$ equipped with the standard dot product. Since
we have assumed that $K$ is algebraically closed, we can find a basis
$e_1, \dots, e_{2n+\delta}$ for $E$ such that $1 = \omega(e_i,
e_{2n+\delta+1-i}) = \pm \omega(e_{2n+\delta+1-i}, e_i)$ (the sign
depending on whether $\omega$ is symmetric or skew-symmetric) for
$i=1,\dots,n+\delta$, and all other pairings are 0. Representing
elements of $G$ as matrices with respect to this ordered basis, we can
take our maximal torus $T$ to be the subgroup of diagonal matrices,
and our Borel subgroup $B$ to be the subgroup of upper triangular
matrices. We identify $(\lambda_1, \dots, \lambda_n) \in \Z^n$ with
the character $\lambda \colon T \to K^*$ given by
\[
\operatorname{diag}(d_1, \dots, d_{2n+\delta}) \mapsto d_1^{\lambda_1}
d_2^{\lambda_2} \cdots d_n^{\lambda_n}.
\]
To be completely explicit, we list the simple roots and the conditions
for a weight to be dominant under this identification in
Table~\ref{weighttable}. 

\begin{table}[ht]
\caption{Roots and weights}
\label{weighttable}
\centering
\begin{tabular}{|c|c|c|c|}
  \hline
  & $\SO(2n+1)$ & $\Sp(2n)$ & $\SO(2n)$ \\
  \hline 
  \tiny{simple roots} & 
  \begin{tabular}{l}
    $\eps_1 - \eps_2, \eps_2 - \eps_3,$ \\ 
    $\dots, \eps_{n-1} - \eps_n, \eps_n$ \end{tabular} & 
  \begin{tabular}{l}
    $\eps_1 - \eps_2, \eps_2 - \eps_3,$ \\ 
    $\dots, \eps_{n-1} - \eps_n, 2\eps_n$ \end{tabular} & 
  \begin{tabular}{l}
    $\eps_1 - \eps_2, \eps_2 - \eps_3, \dots$ \\ 
    $\eps_{n-1} - \eps_n, \eps_{n-1} + \eps_n$ \end{tabular}
  \\
  \hline
  \tiny{dominant weights} & $ \lambda_1 \ge \cdots \ge \lambda_n \ge 0 $ & 
  $\lambda_1 \ge \cdots \ge \lambda_n \ge 0$ & $\lambda_1 \ge
  \cdots \ge \lambda_{n-1} \ge |\lambda_n|$ \\
  \hline
\end{tabular}
\end{table}

We review the relevant details of Weyl's construction for these
representations in characteristic 0. Given $1 \le i < j \le d$, we
have a contraction map
\begin{align*}
  \Psi_{i<j} \colon E^{\otimes d} &\to E^{\otimes d-2}\\
  v_1 \otimes \cdots \otimes v_{d} &\mapsto \omega(v_i, v_j) v_1
  \otimes \cdots \otimes \hat{v}_i \otimes \cdots \otimes \hat{v}_j
  \otimes \cdots \otimes v_{d},
\end{align*}
and we define
\[
E^{\langle d \rangle} = \bigcap_{1 \le i < j \le d} \ker \Psi_{i<j}.
\]
The irreducible polynomial representation of $\GL(E)$ of highest
weight $\lambda$ can be constructed using a Young symmetrizer acting
on $E^{\otimes d}$ where $d = |\lambda|$. The details of this
construction won't be needed, but details can be found in \cite[\S
9.7]{procesi} and \cite[\S 2.2]{weyman}. Let $\bS_\lambda(E)$ denote
such a realization. This notation is compatible with its use in
Section~\ref{invariantsection}. Also, recall that $\ell(\lambda)$ is
the number of nonzero parts of a partition $\lambda$. Finally, we
define
\[
\bS_{[\lambda]}(E) = E^{\langle d \rangle} \cap \bS_\lambda(E).
\]

\begin{proposition} \label{proposition:weylconstruction} The
  intersection $\bS_{[\lambda]}(E)$ is nonzero if and only if
  $\ell(\lambda) \le n$. For $\ell(\lambda) \le n$, we have the
  following cases.
\begin{compactenum}[\rm (a)]
\item Case ${\rm B}_n$: $\bS_{[\lambda]}(E)$ is an irreducible
  representation of $\SO(E)$ with highest weight $\lambda$.
\item Case ${\rm C}_n$: $\bS_{[\lambda]}(E)$ is an irreducible
  representation of $\Sp(E)$ with highest weight $\lambda$.
\item Case ${\rm D}_n$: $\bS_{[\lambda]}(E)$ is an irreducible
  representation of ${\bf O}(E)$. If $\lambda_n = 0$, then
  $\bS_{[\lambda]}(E)$ is an irreducible representation of $\SO(E)$
  with highest weight $\lambda$. Otherwise, if $\lambda_n > 0$, then
  $\bS_{[\lambda]}(E)$ is the direct sum of two irreducible
  representations of $\SO(E)$, one of highest weight $(\lambda_1,
  \dots, \lambda_{n-1}, \lambda_n)$, and one of highest weight
  $(\lambda_1, \dots, \lambda_{n-1}, -\lambda_n)$.
\end{compactenum}
\end{proposition}

\begin{proof} See \cite[\S\S 11.6.3, 11.6.6]{procesi}.
\end{proof}

\begin{remark} \label{dualweights} In types ${\rm B}_n$ and ${\rm
    C}_n$, all representations are self-dual. The same is true for
  type ${\rm D}_n$ when $n$ is even. When $n$ is odd, the dual of a
  representation of type ${\rm D}_n$ with highest weight $(\beta_1,
  \dots, \beta_{n-1}, \beta_n)$ has highest weight $(\beta_1, \dots,
  \beta_{n-1}, -\beta_n)$. This follows from \cite[Proposition
  10.5.3]{procesi}. In particular, we see from
  Proposition~\ref{proposition:weylconstruction} that
  $\bS_{[\lambda]}(E)^* \cong \bS_{[\lambda]}(E)$ in all cases.
\end{remark}

\subsection{Flag quivers.} \label{section:flagquiver} Suppose that the
characteristic is different from 2. Pick positive integers $n$ and
$r$. We construct symmetric quivers $\QQ_{r,n}^+$ and $\QQ_{r,n}^-$ as
follows. The vertices consist of $x^j_i$ and $\tau(x^j_i)$ where
$j=1,\dots,r$ and $i=1,\dots,n$ along with an additional vertex $u =
\tau(u)$. For notation we use $x^j_{n+1}$ to denote $u$ for any
$j=1,\dots,r$. The action of the involution $\tau$ is suggested by the
notation and $s(u)$ is given by the superscript of
$\QQ^{\pm}_{r,n}$. For each $j=1,\dots,r$ and each $i=1,\dots,n$, we
have an arrow $x^j_i \xrightarrow{a_i^j} x^j_{i+1}$ and
$\tau(x^j_{i+1}) \xrightarrow{\tau(a_i^j)} \tau(x^j_i)$. Again, the
action of $\tau$ is suggested by the notation. Furthermore, we impose
the relations $\tau(a_n^j) a_n^j = 0$ for $j=1,\dots,r$. We have drawn
a diagram in the case $r=3$:
\[
\xymatrix{ x^1_1 \ar[r]^-{a^1_1} & x^1_2 \ar[r]^-{a^1_2} & \cdots
  \ar[r]^-{a^1_{n-1}} & x^1_n \ar[dr]^-{a^1_n} & & \tau(x^1_n)
  \ar[r]^-{\tau(a^1_{n-1})} & \cdots \ar[r]^-{\tau(a^1_2)} &
  \tau(x^1_2) \ar[r]^-{\tau(a^1_1)} & \tau(x^1_1) \\
  x^2_1 \ar[r]^-{a^2_1} & x^2_2 \ar[r]^-{a^2_2} & \cdots
  \ar[r]^-{a^2_{n-1}} & x^2_n \ar[r]^-{a^2_n} & u = \tau(u)
  \ar[r]^-{\tau(a^2_n)} \ar[ur]^-{\tau(a^1_n)} \ar[dr]^-{\tau(a^3_n)}
  & \tau(x^2_n) \ar[r]^-{\tau(a^2_{n-1})} & \cdots
  \ar[r]^-{\tau(a^2_2)} & \tau(x^2_2) \ar[r]^-{\tau(a^2_1)} & \tau(x^2_1) \\
  x^3_1 \ar[r]^-{a^3_1} & x^3_2 \ar[r]^-{a^3_2} & \cdots
  \ar[r]^-{a^3_{n-1}} & x^3_n \ar[ur]^-{a^3_n} & & \tau(x^3_n)
  \ar[r]^-{\tau(a^3_{n-1})} & \cdots \ar[r]^-{\tau(a^3_2)} &
  \tau(x^3_2) \ar[r]^-{\tau(a^3_1)} & \tau(x^3_1) }.
\]
We call $\QQ^+_{r,n}$ the {\bf orthogonal flag quiver} and
$\QQ^-_{r,n}$ the {\bf symplectic flag quiver}. We define a dimension
vector $\beta^\delta$ by $\beta^\delta(x^j_i) = i$ for $j=1,\dots,r$
and $i=1,\dots,n$, and define $\beta^\delta(u) = 2n+\delta$. Since
$\delta = 0$ in the symplectic case, we will usually write $\beta =
\beta^0$ in this case. We will use $\QQ_{r,n}$ to denote the
underlying quiver with relations.

\begin{proposition} The quiver $\QQ_{r,n}$ has global dimension $2$.
\end{proposition}

\begin{proof} It is enough to show that every simple module $S_x$ has
  a projective resolution of length at most 2. We write down the
  resolutions and leave the verification to the reader. For $u$, we
  have
  \[
  0 \to \bigoplus_{j=1}^r P_{\tau(x_n^j)} \to P_u \to S_u \to 0.
  \]
  For each $j=1,\dots,r$ and $i=1,\dots,n$ we have (with the
  convention that $P_{\tau(x^j_0)} = 0$):
  \begin{align*}
    0 \to P_{\tau(x^j_{i-1})} \to P_{\tau(x^j_i)} \to S_{\tau(x^j_i)}
    \to 0& \\
    0 \to P_{x^j_{i+1}} \to P_{x^j_i} \to S_{x^j_i} \to 0& \quad (1
    \le i \le n-1)\\
    0 \to P_{\tau(x^j_n)} \to P_u \to P_{x^j_n} \to S_{x^j_n} \to
    0&. \qedhere
  \end{align*}
\end{proof}

From now on, we replace all instances of orthogonal groups by special
orthogonal groups in the definition of semi-invariants. In effect, we
are ignoring the action of ${\bf O}(V) / \SO(V)$ on
semi-invariants. Hence we do not have to worry about irreducible
components of $\SRep(Q/I, \beta)$ not being closed under the action of
${\bf G}(Q, \beta)$.

A symmetric representation $V \in \SRep(\QQ^+_{r,n}, \beta^\delta)$ is
given by the data of $(2n+\delta)$-dimensional vector space $V(u)$
equipped with a nondegenerate symmetric form, in addition to arbitrary
vector spaces $V(x^j_i)$ of dimension $i$ and arbitrary linear maps
$V_{a^j_i}$ for $j=1,\dots,r$ and $i=1,\dots,n$. The relation
$\tau(a^j_n)a^j_n = 0$ is equivalent to saying that the image of
$V(x^j_n)$ under $V_{a^j_n}$ is an isotropic subspace. This can be
seen by picking a hyperbolic basis for $V(u)$. There is a similar
interpretation for a symmetric representation in $\SRep(\QQ^-_{r,n},
\beta)$.

Let $Z' \subset \hom(V(x^j_n), V(u))$ be the subvariety of maps whose
image is an isotropic subspace. Let $Y$ be the Grassmannian of
$n$-dimensional isotropic subspaces of $V(u)$. Then $Y$ is equipped
with a trivial vector bundle $V(u) \times Y$ and a tautological
subbundle $\RR \subset V(u) \times Y$ given by $\{(x,W) \mid x \in
W\}$. We can also form the vector bundle $Z = \hom(V(x^j_n), \RR) =
V(x^j_n)^* \otimes \RR$. Then $Z$ consists of pairs $(\phi, W)$ where
$\phi \colon V(x^j_n) \to W$, so there is a natural projection $\pi
\colon Z \to \hom(V(x^j_n), V(u))$ whose image is $Z'$.

\begin{proposition} \label{proposition:desingularization} The map $\pi
  \colon Z \to Z'$ is a projective birational morphism.
\end{proposition}

\begin{proof} Since $Y$ is a projective variety, the projection $V(u)
  \times Y \to V(u)$ is projective, so the same is true for the
  restriction $\pi$. The set of injective maps in $Z'$ is open and
  dense, and there is a uniquely defined inverse on this open set.
\end{proof}

\begin{corollary} The open subset of $Z'$ consisting of injective maps
  is nonsingular.
\end{corollary}

In the symplectic case and the odd orthogonal case, $Y$ is an
irreducible variety since it has a transitive action of $\Sp(V(u))$
and $\SO(V(u))$, respectively. So $Z'$ is also irreducible in these
cases. In the even orthogonal case, $\SO(V(u))$ does not act
transitively on $Y$. However ${\bf O}(V(u))$ does act transitively, so
$Y$ has two connected components. To describe them, fix a maximal
isotropic subspace $W \subset V(u)$. Then one component consists of
subspaces whose intersection with $W$ has even dimension, and the
other component consists of subspaces whose intersection with $W$ has
odd dimension. In particular, the subvariety $Z' \subset
\hom(V(x^j_n), V(u))$ of isotropic maps has two irreducible
components. Each component has a dense open subset consisting of
injective maps $\phi$ and the components are distinguished by whether
$\dim(\phi(V(X_n)) \cap W)$ is even or odd. Their intersection
consists of the non-injective maps.

\begin{proposition} \label{proposition:components} The varieties
  $\SRep(\QQ^+_{r,n}, \beta^1)$ and $\SRep(\QQ^-_{r,n}, \beta)$ are
  irreducible. The variety $\SRep(\QQ^+_{r,n}, \beta^0)$ has $2^r$
  irreducible components. Each of the components is faithful (see
  Section~\ref{section:semiinvariantsQ/I}).
\end{proposition}

\begin{proof}
  The varieties $\hom(V(x^j_i), V(x^j_{i+1}))$ are irreducible for
  $j=1,\dots,r$ and $i=1,\dots,n-1$, so the statements about the
  irreducible components follow from the preceding discussion. The
  faithfulness of the components in $\SRep(\QQ_{r,n}^+, \beta^0)$
  follows from their explicit description in the preceding
  discussion. 
\end{proof}

\begin{corollary} The restriction of $\pi$ for each irreducible
  component $X$ of $\SRep(\QQ^{\pm}_{r,n}, \beta^\delta)$ is a rational
  desingularization. In particular, they are normal varieties.
\end{corollary}

This follows from the results of \cite[Chapter 5]{weyman}. We won't
need it, so we omit the details.

\begin{proposition} \label{proposition:evencase} Let $V$ be a
  representation of $\QQ_n^{\pm}$. Then $\pdim V \le 1$ if and only if
  the maps $V(x^j_n) \to V(u)$ are injective for
  $j=1,\dots,r$. Dually, $\idim V \le 1$ if and only if the maps $W(u)
  \to W(\tau(x^j_n))$ are surjective for $j=1,\dots,r$.
\end{proposition}

\begin{proof} Throughout the proof, let $Q/I$ denote either the quiver
  with relations $\QQ_{r,n}^+$ or $\QQ_{r,n}^-$. 

  First we prove that if the maps $V(x^j_n) \to V(u)$ are injective,
  then $\pdim V \le 1$. Let $V'$ be the submodule generated by the
  $V(x^j_n)$ for $j=1,\dots,r$. Then $\pdim V/V' \le 1$ since $V/V'$
  is supported on a quiver without relations. So $\pdim V \le 1$ if we
  can prove that $\pdim V' \le 1$. Furthermore, the submodule $V''$ of
  $V'$ generated by the $V'(\tau(X_n))$ is supported in a quiver
  without relations, so it is enough to show that $\pdim W \le 1$
  where $W = V'/V''$.

  Set $P_0 = \bigoplus_{j=1}^r P_{x^j_n}^{\oplus \alpha(x^j_n)}$. To
  say that the kernel $P_1$ of the surjection $P_0 \to W \to 0$ is
  projective, it is enough to say that the submodule of $P_0$
  generated by $P_1(u)$ is projective, which in our situation amounts
  to saying that the maps $P_1(\tau(a^j_n))$ are injective. Pick a
  basis $v_1, \dots, v_N$ for $P_1(u)$. Each basis vector $v_i$ can be
  written as $v_i = v^1_i + \cdots + v^r_i$ where $v^j_i \in
  P_{x^j_n}^{\oplus \alpha(x^j_n)}(u)$, and its image in
  $P_1(\tau(x^j_n))$ is $v_i - v^j_i$. Suppose that the vectors $v_1 -
  v_1^j, \dots, v_N - v_N^j$ are linearly dependent. Since $v_1,
  \dots, v_N$ are linearly independent, this means that some nonzero
  linear combination of the $v_i$ is in $P_{x^j_n}^{\oplus
    \alpha(x^j_n)}(u)$. Hence the map $P_{x^j_n}^{\oplus
    \alpha(x^j_n)}(u) \to W(u)$ coming from $P_0 \to W$ has a nonzero
  kernel, which contradicts that $W(x^j_n) \to W(u)$ is
  injective. Therefore the images of $v_1, \dots, v_N$ under each
  $P_1(\tau(a^j_n))$ are linearly independent, so we are done.

  On the other hand, if the map $V(x^j_n) \to V(u)$ is not injective,
  then $P_0$ will contain a direct summand $P_{x^j_n}$ such that
  $P_{x^j_n}(u)$ will be in the kernel of $P_0 \to V \to 0$. This
  requires that $P_1$ contain a summand $P_u$, and the restriction of
  the map $P_1 \to P_0$ to this summand will not be injective.

  The dual statement about injective dimension is proved in a similar
  manner, or can be obtained by applying the duality functor
  $\hom_K(-,K)$.
\end{proof}

\subsection{Back to semi-invariants.}

Suppose that the characteristic is 0. Now we calculate the space of
semi-invariants of $\SRep(\QQ^{\pm}_{r,n}, \beta^\delta)$. First, we
need a better understanding of what the coordinate ring of
$\SRep(\QQ^{\pm}_{r,n}, \beta^\delta)$ is. Since
$\SRep(\QQ^{\pm}_{r,n}, \beta^\delta)$ is an $r$-fold product of
$\SRep(\QQ^{\pm}_{1,n}, \beta^\delta)$, it is enough to calculate the
coordinate ring in the case $r=1$. Let $R$ denote this variety in the
case $r=1$ and set $x_i = x^1_i$ and let $Z \subset \hom(V(x_n),
V(u))$ denote the subvariety of maps whose image is an isotropic
subspace. Then $R$ is a product of an affine space with $Z$, so we
just need to describe the coordinate ring of $Z$.

By the Cauchy identity \eqref{eqn:cauchy}, the polynomial functions on
$\hom(V(x_n), V(u))$ are given by
\[
\Sym(V(x_n) \otimes V(u)^*) = \bigoplus_{\lambda} \bS_\lambda(V(x_n))
\otimes \bS_\lambda(V(u)^*).
\]
Imposing the relation that the image of the map must be an isotropic
subspace means that after applying the Schur functor $\bS_\lambda$,
the image of $\bS_\lambda(V(x_n)) \to \bS_\lambda(V(u))$ must be in
the kernel of each contraction map as described in
Section~\ref{sec:classicalrepn}. So the coordinate ring of $Z$ is a
quotient of
\[
\bigoplus_\lambda \bS_\lambda(V(x_n)) \otimes \bS_{[\lambda]}(V(u)^*).
\]
In fact, we have equality. To see this, note that $\bS_\lambda(V(x_n))
\otimes \bS_\lambda(V(u)^*)$ is nonzero on some map $\phi$. We can
conjugate $\phi$ by $\GL(V(u))$ to $g\phi$ so that the image of
$g\phi$ is an isotropic subspace. Hence $\bS_\lambda(V(x_n)) \otimes
\bS_\lambda(V(u)^*)$ will also be nonzero on $g\phi$, and
$\bS_\lambda(V(x_n)) \otimes \bS_{[\lambda]}(V(u)^*)$ has to be
nonzero on $g\phi$. Furthermore, letting $G$ be either $\GL(V(x_n))
\times {\bf O}(V(u))$ or $\GL(V(x_n)) \times \Sp(V(u))$ depending on
which case we are in, $\bS_\lambda(V(x_n)) \otimes
\bS_{[\lambda]}(V(u)^*)$ is an irreducible $G$-module and $Z$ is a
$G$-invariant subvariety of $\hom(V(x_n), V(u))$, so no function of
$\bS_\lambda(V(x_n)) \otimes \bS_{[\lambda]}(V(u)^*)$ can be
identically zero on $Z$. 

By Proposition~\ref{proposition:components}, the variety $R$ is
irreducible in types B and C, but has 2 components $R^+$ and $R^-$ in
type D. We want to describe the coordinate rings for $R^+$, $R^-$ and
$R^0 = R^+ \cap R^-$. So now we assume that $V(u)$ has a nondegenerate
symmetric bilinear form and has dimension $2n$. Recall from
Proposition~\ref{proposition:weylconstruction} that the representation
$\bS_{[\lambda]}(V(u))$ is irreducible as an ${\bf O}(V(u))$-module,
but that when $\lambda_n > 0$, it splits into the direct sum of two
representations when considered as an $\SO(V(u))$-module. Call these
two summands $\bS_{[\lambda]^\pm}(V(u))$. For notational purposes, the
symbols $[\lambda]^\pm$ mean $[\lambda]$ in the case that $\lambda_n =
0$. Fix an element $g \in {\bf O}(V(u)) \setminus
\SO(V(u))$. Conjugation by $g$ is an outer automorphism of $\SO(V(u))$
that transforms $\bS_{[\lambda]^+}(V(u))$ into
$\bS_{[\lambda]^-}(V(u))$ and vice versa. Furthermore, while $R^+$ and
$R^-$ are preserved by $\SO(V(u))$, the element $g$ swaps the two of
them and preserves $R^0$. The points in $R^0$ consist of non-injective
maps $V(x_n) \to V(u)$. So when $\lambda_n > 0$, the functions in
$\bS_{\lambda}(V(x_n)) \otimes \bS_{[\lambda]}(V(u)^*)$ vanish on
$R^0$. 

We claim that the ``positive'' and the ``negative'' representations of
$\SO(V(u))$ form the coordinate rings of the two separate components
of $R$. If not, suppose that both $\bS_\lambda(V(x_n)) \otimes
\bS_{[\lambda]^+}(V(u)^*)$ and $\bS_\mu(V(x_n)) \otimes
\bS_{[\mu]^-}(V(u)^*)$ are nonzero on the component $R^+$ where
$\lambda_n > 0$ and $\mu_n > 0$ and choose highest weight vectors
$f_\lambda$ and $f_\mu$ in both. Since $R^+$ is irreducible,
$f_\lambda f_\mu \ne 0$ and has weight $(\lambda + \mu, [\lambda]^+ +
[\mu]^-)$ when $n$ is even, and has weight $(\lambda + \mu,
[\lambda]^- + [\mu]^+)$ when $n$ is odd
(Remark~\ref{dualweights}). But there are no representations in $K[Z]$
with this highest weight, which is a contradiction. Therefore, setting
\[
A = \bigoplus_{\lambda^1, \dots, \lambda^{n-1}} \bigotimes_{1 \le i
  \le n-1} \bS_{\lambda^i}(V(x_i)) \otimes
\bS_{\lambda^i}(V(x_{i+1})^*),
\]
we have
\begin{align*}
  K[R^+] &= A \otimes \bigoplus_\lambda \bS_{\lambda}(V(x_n)) \otimes
  \bS_{[\lambda]^+}(V(u)^*),\\
  K[R^-] &= A \otimes \bigoplus_\lambda \bS_{\lambda}(V(x_n)) \otimes
  \bS_{[\lambda]^-}(V(u)^*),\\
  K[R^0] &= A \otimes \bigoplus_{\lambda,\ \lambda_n = 0}
  \bS_{\lambda}(V(x_n)) \otimes \bS_{[\lambda]}(V(u)^*).
\end{align*}
There is no inherent way of distinguishing the two components or the
representations $\bS_{[\lambda]^{\pm}}(V(u)^*)$, so the formulas above
should be taken as sign conventions.

\begin{proposition} \label{prop:coordinatering}
  The coordinate ring of $\SRep(\QQ^{\pm}_{r,n}, \beta^\delta)$ is
  \[
  \bigoplus_{\ul{\lambda}} \bigg( \bigotimes_{\substack{ 1 \le j \le
      r\\ 1 \le i \le n-1}} \bS_{\ul{\lambda}(x^j_i)}(V(x^j_i))
  \otimes \bS_{\ul{\lambda}(x^j_i)}(V(x^j_{i+1})^*) \otimes
  \bigotimes_{1 \le j \le r} \bS_{\ul{\lambda}(x^j_n)}(V(x^j_n))
  \otimes \bS_{[\ul{\lambda}(x^j_n)]}(V(u)^*) \bigg),
  \]
  where $\ul{\lambda}$ ranges over all partition-valued functions of
  the set $\{x^j_i\}_{1 \le i \le n}^{1 \le j \le r}$. The coordinate
  ring of an irreducible component in the type ${\rm D}$ case is given
  by a choice of pluses and minuses $\ul{\eps} \colon \{x^j\} \to
  \{+,-\}$ to add to the $[\ul{\lambda}(x^j_n)]$. We use the notation
  $\SRep(\QQ^+_{r,n}, \beta)^{\ul{\eps}}$ to denote the corresponding
  component.
\end{proposition}

Given this description, we can now calculate the space of
semi-invariants in terms of representations of the corresponding
classical group. This is essentially the same as \cite[\S
3]{saturation} but we have tried to provide more details.

\begin{lemma} \label{lemma:partitionform} If the direct summand
  corresponding to the function $\ul{\lambda}$ of
  Proposition~\ref{prop:coordinatering} contains a nonzero
  $\SG(\QQ^{\pm}_{r,n}, \beta^\delta)$-invariant, then there exist
  numbers $\sigma(x^j_i)$ for $i=1,\dots,n$ and $j=1,\dots,r$ such
  that
  \begin{align} \label{weightform} \ul{\lambda}(x^j_i)' =
    (1^{\sigma(x^j_1)}, 2^{\sigma(x^j_2)}, \dots, i^{\sigma(x^j_i)}).
\end{align}
Furthermore, $\sigma$ is the symmetric weight of the corresponding
semi-invariant.
\end{lemma}

\begin{proof} We prove this by induction on $i$. For $i=1$, we have
  that $\bS_{\ul{\lambda}(x^j_1)}(V(x^j_1))$ is a nonzero
  $\SL(V(x^j_1))$-invariant if and only if $\ul{\lambda}(x^j_1)' =
  (1^{\sigma(x^j_1)})$ for some $\sigma(x^j_1) \ge 0$ since $\dim
  V(x^j_1) = 1$.

  For general $i$, write $\mu' = (1^{\sigma(x^j_{i-1})},
  2^{\sigma(x^j_{i-2})}, \dots, (i-1)^{\sigma(x^j_1)})$. Then
  \[
  \bS_{\ul{\lambda}(x^j_{i-1})}(V(x^j_{i})^*) \otimes
  \bS_{\ul{\lambda}(x^j_i)}(V(x^j_i)) \cong (\bigwedge^i
  V(x^j_i))^{\otimes -\Sigma } \otimes \bS_{\mu}(V(x^j_i)) \otimes
  \bS_{\ul{\lambda}(x^j_i)}(V(x^j_i))
  \]
  as $\GL(V(x^j_i))$-representations, where $\Sigma = \sigma(x^j_1) +
  \cdots + \sigma(x^j_{i-1})$ by Proposition~\ref{prop:dualschur}. We
  see that if there is a nonzero $\SL(V(x^j_i))$-invariant, then $k$
  must appear as a part of $\ul{\lambda}(x^j_i)'$ with multiplicity
  $\sigma(x^j_k)$ for $k=1,\dots,i-1$, and $i$ can appear with
  arbitrary multiplicity $\sigma(x^j_i)$ by
  Proposition~\ref{prop:SLinvariants}. Hence $\ul{\lambda}(x^j_i)$ is
  also in the form \eqref{weightform}, and the action of
  $\GL(V(x^j_i))$ on the $\SL(V(x^j_i))$-invariants is via the
  $\sigma(x^j_i)$th power of the determinant again by
  Proposition~\ref{prop:SLinvariants}. 
\end{proof}

\begin{proof}[Proof of Theorem~\ref{bcdsaturation}] 
  Let $Q/I$ denote either $\QQ^{+}_{r,n}$ or $\QQ^-_{r,n}$ depending
  on which group we are interested in. If $(V_{N\lambda^1} \otimes
  \cdots \otimes V_{N\lambda^r})^G \ne 0$, then by
  Lemma~\ref{lemma:partitionform}, there is a corresponding nonzero
  weight space $\SSI(Q/I, \beta^\delta)^{\ul{\eps}}_{N\sigma}$ for
  some irreducible component of $\SRep(Q/I, \beta^\delta)$ (in type B
  or C, there is only one component, in which case the superscript
  $\ul{\eps}$ means nothing). Since this space is spanned by Pfaffian
  semi-invariants by Proposition~\ref{proposition:pfaffianrelations},
  there must be a nonzero determinantal semi-invariant in $\SSI(Q/I,
  \beta)^{\ul{\eps}}_{2N\sigma}$. There is not a unique solution
  $\alpha$ for $2N\sigma = \sigma_{N\alpha}$ since $\sigma(u)$ is not
  defined, so the variable $\alpha(u)$ is not yet determined, and the
  variables $\alpha(\tau(X_i))$ are $\alpha(u)$ plus some
  constant. When we calculate $\langle \alpha, \beta^\delta
  \rangle_I$, the coefficient of $\alpha(u)$ is $2n+\delta$
  (independent of what $r$ is), so using the fact that $\langle
  \alpha, \beta^\delta \rangle_I = 0$, we can solve for $\alpha(u)$
  uniquely.

  So there is a representation $V \in \Rep(Q/I, N\alpha)$ such that
  $\ol{c}^V$ is nonzero on $\SRep(Q/I,
  \beta^\delta)^{\ul{\eps}}$. This implies that $\ol{c}^V$ is nonzero
  on $\Rep(Q/I, \beta^\delta)$, so we can conclude that $\pdim V \le
  1$ by Theorem~\ref{theorem:SIrelations}. Since $\pdim V \le 1$,
  Remark~\ref{remark:weights} gives 
  \[
  0 \ge -\sigma^\circ(\tau(x^j_n)) = \langle \alpha, \eps_{\tau(x^j_n)}
  \rangle_I = \alpha(\tau(x^j_n)) - \alpha(u) + \alpha(x^j_n),
  \]
  so $\alpha(u) \ge \alpha(x^j_n) + \alpha(\tau(x^j_n))$ for
  $j=1,\dots,r$. Since $\pdim V \le 1$ and $\ol{c}^V \ne 0$, we have
  that $\ext^1(V, -)$ vanishes generically on $\SRep(Q/I,
  \beta^\delta)^{\ul{\eps}}$. Let $W$ be a general element of
  $\SRep(Q/I, \beta^\delta)^{\ul{\eps}}$. This implies that $\langle
  N\alpha, \beta' \rangle_I \ge 0$ for all dimension vectors $\beta'$
  of factor modules of $W$. To see this, let $W' \subseteq W$ be a
  submodule of dimension $\beta^\delta - \beta'$. Since $\pdim V \le
  1$, we have $\ext^2(V, W') = 0$, so $\ext^1(V, W/W') = 0$ since
  $\ext^1(V, W) = 0$. This gives $\langle N\alpha, \beta' \rangle_I =
  \dim \hom(V, W/W') \ge 0$, which is the desired inequality.

  So $\langle \alpha, \beta' \rangle_I \ge 0$ for all dimension
  vectors $\beta'$ of factor modules of $W$. By
  Proposition~\ref{proposition:evencase}, $\idim W \le 1$ and $\pdim W
  \le 1$, and the same is true for a general representation of
  dimension $\alpha$ since we have shown that $\alpha(u) \ge
  \alpha(x^j_n) + \alpha(\tau(x^j_n))$ for $j=1,\dots,r$. Also, if
  $\gamma$ is the generic rank of a map between a general
  representation of dimension $\alpha$ and $W$, then $\gamma(u) \ge
  \gamma(x^j_n) + \gamma(\tau(x^j_n))$ for $j=1,\dots,r$: the image
  $U$ of such a generic map satisfies $\idim U \le 1$ (being a
  quotient of a module with injective dimension at most 1) and $\pdim
  U \le 1$ (being a submodule of $W$), so $\dim \Gr(\gamma, W)' =
  \langle \gamma, \beta - \gamma \rangle_I$ by
  Proposition~\ref{proposition:symmetricgrassmannian}. So the
  hypotheses of Lemma~\ref{lemma:induction} are satisfied and
  Theorem~\ref{theorem:vanishingext} gives $V' \in \Rep(Q/I, \alpha)$
  such that $\ext^1_{Q/I}(V', W) = 0$. Hence $\ext^1_{Q/I}(V', -)$
  vanishes generically on $\SRep(Q/I,
  \beta^\delta)^{\ul{\eps}}$. Finally, this means that
  \[
  (V_{2\lambda^1} \otimes \cdots \otimes V_{2\lambda^r})^G = \SSI(Q/I,
  \beta^\delta)^{\ul{\eps}}_{2\sigma} \ne 0. \qedhere
  \]
\end{proof}

\begin{example} Let $n = 2$ and $r=3$ so that 
  \[
  \beta^\delta = \begin{matrix} 1 & 2 & & 2 & 1 \\ 1 & 2 & 4+\delta &
    2 & 1 \\ 1 & 2 & & 2 & 1 \end{matrix}
  \]
  The product $V_{(2)} \otimes V_{(2)} \otimes V_{(4)}$ has a
  $G$-invariant, and from Lemma~\ref{lemma:partitionform}, this gives
  a symmetric weight $\sigma$. The corresponding non-symmetric weight
  is
  \[
  \sigma^\circ = \begin{matrix} 1 & 0 & & 0 & -1 \\ 1 & 0 & ? & 0 & -1
    \\ 2 & 0 & & 0 & -2 \end{matrix}.
  \]
  Note that we have to halve the weights when we think of $\sigma$ as
  a non-symmetric weight $\sigma^\circ$. The question mark means that
  there is nothing assigned to the symmetric weight at that vertex,
  and hence the non-symmetric value cannot be determined yet. Using
  the fact that $\langle \alpha, \eps_x \rangle_I = \sigma^\circ_x$
  gives us that
  \[
  \alpha = \begin{matrix} 1 & 1 & & a-1 & a-2 \\ 1 & 1 & a &
    a-1 & a-2 \\ 2 & 2 & & a-2 & a-4 \end{matrix}
  \]
  for some value of $a$. Now $0 = \langle \alpha, \beta \rangle_I =
  (a-4)(4 + \delta)$ means that $a = 4$. 

  The product $V_{(1)} \otimes V_{(1)} \otimes V_{(2)}$ also has a
  $G$-invariant, but the corresponding symmetric weight $\sigma'$ is
  not divisible by 2, which means that we cannot express $\sigma'$ as
  a non-symmetric weight. This means that the $\sigma'$-weight space
  is spanned by Pfaffian semi-invariants but does not contain any
  determinantal semi-invariants. 
\end{example}

\small \noindent Steven V Sam, Department of Mathematics,
Massachusetts Institute of Technology, Cambridge, MA 02139 \\
{\tt ssam@math.mit.edu}, \url{http://math.mit.edu/~ssam/}

\end{document}